\documentclass[11pt]{article}

\usepackage{amssymb}
\usepackage{amsmath}
\usepackage{amsthm}
\usepackage[english]{babel}
\usepackage[latin1]{inputenc}
\usepackage{latexsym}
\usepackage{mathrsfs}
\usepackage{bbm}
\usepackage{graphicx}
\usepackage[hidelinks]{hyperref}

\numberwithin{equation}{section}

\newtheorem{thm}{Theorem}[section]
\newtheorem{prop}[thm]{Proposition}
\newtheorem{defn}[thm]{Definition}
\newtheorem{lemma}[thm]{Lemma}
\newtheorem{corol}[thm]{Corollary}

\newtheorem{ipo}[thm]{Assumption}

\theoremstyle{definition}
\newtheorem{remark}[thm]{Remark}

\newcommand{\nn}{\mathbb{N}}
\newcommand{\rr}{\mathbb{R}}
\newcommand{\ppp}{\mathbb{P}_{tpz}}
\newcommand{\eee}{\mathbb{E}_{tpz}}
\newcommand{\dd}{\mathcal{D}}
\newcommand{\aaa}{\mathcal{A}_{t}}
\newcommand{\aav}{\mathcal{A}_{tz}^\textup{adm}}

\newcommand{\eps}{\varepsilon}

\usepackage[a4paper,top=3cm,bottom=3.5cm,left=3cm,right=3cm]{geometry}

\title{\textbf{Optimal exercise of swing contracts\\ in energy markets: an integral constrained\\ stochastic optimal control problem}}

\author{Matteo Basei\footnote{Department of Mathematics, University of Padova, via Trieste 63, 35121 Padova, Italy. Emails: basei@math.unipd.it, acesar@math.unipd.it, vargiolu@math.unipd.it.} \qquad\, Annalisa Cesaroni\footnotemark[1] \qquad\, Tiziano Vargiolu\footnotemark[1]}

\date{\empty}

\begin{document}

\maketitle

\vspace{0.4cm}
\begin{abstract}
%\begin{quote}
We characterize the value of swing contracts in continuous time as the unique viscosity solution of a Hamilton-Jacobi-Bellman equation with suitable boundary conditions. The case of contracts with penalties is straightforward, and in that case only a terminal condition is needed. Conversely, the case of contracts with strict constraints gives rise to a stochastic control problem with a nonstandard state constraint. We approach this
%ABSTRACT. Motivated by the problem of the optimal exercise of swing contracts in energy markets, we consider a class of stochastic optimal control problems with integral constraints on the controls. We study the
problem by a penalty method: we consider a general constrained problem and approximate the value function with a sequence of value functions of appropriate unconstrained problems with a penalization term in the objective functional. Coming back to the case of swing contracts with strict constraints, we finally characterize the value function as the unique viscosity solution with polynomial growth of the HJB equation subject to appropriate boundary conditions.
%\end{quote}
\end{abstract}

\vspace{0.4cm}
\vspace{0.7cm}

\section{Introduction}

Energy is traded in financial markets, in its various forms (electricity, coal, gas, oil, etc.), mainly through two types of contracts, namely forwards and swings. Forward contracts are obligations between two parts to exchange some amount of energy, in a specified form (electricity or some fuel) and for a prespecified amount of money: once settled, this contract is strictly binding for both the parts, giving no flexibility to them. Conversely, swing contracts give a certain amount of flexibility to the buyer, while also giving the seller a certain guarantee that a minimum quantity of energy will be bought. This is due to the fact that energy storage is costly in the case of fuels and almost impossible in the case of electricity; moreover, energy markets are influenced by many elements (peaks in consumes related to sudden weather changes, breakdowns in power plants, financial crises, etc.). As a consequence, the price of energy is subject to remarkable fluctuations, so that flexibility is much welcomed by contract buyers.

%To hedge against the risk of sudden price rises, swing contracts with penalties and swing contracts with strict constraints are traded in the market. The owner of a swing contract with penalties has the right, during a fixed amount of time, to buy energy at a fixed price, but there is a bound in the purchase intensity and a penalty must be paid whether the total bought quantity does not belong to a fixed range. The9 problem of optimally exercising such a contract can be approached through classical results in control theory (see Section \ref{CAPopzioniswingpenal}). In swing contracts with strict constraints the total purchased energy must strictly fall within the fixed range: we get problem (\ref{intro1}, to which classical theory does not apply, as stressed before.

The flexibility in swing contracts is implemented in this way (we here follow the approach in \cite{BLN} and model the contract in continuous time): for a fixed contract maturity $T$ (usually one or several years), the buyer can choose, at each time $s \in [0,T]$, to buy a marginal amount of energy $u(s) \in [0,\bar u]$ at a prespecified strike price $K$, thus realizing a marginal profit (or loss) equal to $u(s) (P(s) - K)$, where $P(s)$ is the spot price of that kind of energy. This gives to the buyer the potential profit (or loss)
\begin{equation*}
\int_0^T e^{-rs} u(s) (P(s) - K)\ ds, 
\end{equation*}
with $r > 0$ the risk-free interest rate.

However, the energy seller usually wants the total amount of energy $Z(T) = \int_0^T u(s)\ ds $ to lie between a minimum and a maximum quantity, that is $Z(T) \in [m,M]$. This is implemented in two main ways. The first way is to impose penalties when $Z(T) \notin [m,M]$, i.e.~to make the buyer pay a penalty $\widetilde \Phi(P(T),Z(T))$, where $\widetilde \Phi(p,z)$ is a contractually fixed function, null for $z \in [m,M]$ and usually convex in $z$. The second way is to impose the constraint $Z(T) \in [m,M]$ to be satisfied strictly, i.e.~to force the buyer to withdraw the minimum cumulative amount of energy $m$ and to stop giving the energy when the maximum $M$ has been reached.

We are interested in the problem of optimally exercising a swing option, in both the cases. This problem can be modelled as a continuous time stochastic control problem: our aim is to study the corresponding value function and to characterize it as the unique viscosity solution of the related Hamilton-Jacobi-Bellman (HJB) equation.   In doing this, we feel we are filling a gap in literature, where swing contracts are treated either in discrete time \cite{BBP,Barrera,EFRV,JRT} via the dynamic programming principle and Bellman equations or in continuous time \cite{BLN,BLN2} only by reporting a verification theorem for a smooth solution of the HJB equation, without reporting existence or uniqueness results for that. Besides, we also extend the approach in \cite{BLN,BLN2}, which only treat the case $m = 0$, to the case when $m > 0$, which is the most relevant case in practical applications (in fact, \cite{Loland} reports that typically $m \in [0.8M,M]$). We must also cite \cite{ct}, where swing contracts in continuous time are treated in continuous time with multiple stopping techniques, but for the actual computation for swing contracts in finite time reduce the problem to a discrete time approximation. 

In the case of swing contracts with penalties, we get a standard stochastic control problem, as the maximization of the final expected payoff for a buyer entering in the contract at a generic time $t \in [0,T]$ is given by
\begin{equation}
\widetilde V(t,p,z)= \sup_{u \in \aaa} \eee \left[\int_t^T e^{-rs}(P(s)-K)u(s)\ ds - e^{-rT} \widetilde\Phi(P(T),Z(T))\right], 
\end{equation}
with $(t,p,z) \in [0,T]\times \rr^2$, where $\aaa$ is the set of $[0, \bar u]$-valued progressively measurable processes $u=\{u(s)\}_{s\in[t,T]}$. Thus, in this case classical theory (see \cite{DaLioLey2,EV,FS}) can be applied: see Section \ref{CAPopzioniswingpenal}.

Conversely, swing contracts with strict constraints give rise to a stochastic control problem with a nonstandard state constraint:
%Thus, for the nontrivial case $m > 0$, the constraint $Z(T) \geq m$ gives rise to the nonstandard stochastic control problem
\begin{equation}
\label{intro2}
V(t,p,z) = \sup_{u \in \aav}  \eee \biggr[ \int_t^T e^{-r(s-t)} (P(s) - K) u(s)\ ds \biggr],
\end{equation}
with $(t,p,z)$ in a suitable domain $\dd \subseteq [0,T]\times \rr^2$, where $\aav$ is the set of processes $u \in \aaa$ such that $\ppp\text{-a.s.}$ $Z(T)\in[m,M]$. Due to the presence of the constraint on $Z^{t,z;u}(T)$, here classical theory does not apply.

This motivated us to consider in Section \ref{CAPcontrottvinc} a more general class of integral constrained stochastic problems in the form 
\begin{equation}
\label{intro3}
V(t,p,z)= \sup_{u \in \aav}  \eee \biggr[ \int_t^T e^{-r(s-t)}L(s,P(s),Z(s),u(s))ds \\
+ e^{-r(T-t)}\Phi(P(T),Z(T))\biggr].
\end{equation}
of which Equation (\ref{intro2}) is a particular case. 
For precise definitions and assumptions we refer to Section \ref{COV}; in particular, we ask a technical hypothesis (Assumption \ref{assrho}) to hold. Since a straightforward modification of standard proofs is not possibile, we follow a different approach: we show that in a set $\widetilde \dd \subseteq \dd$ the function $V$ in Equation (\ref{intro3}) is the limit of the value functions $V^c$ of suitable unconstrained problems, where the constraint has been substituted by an appropriate penalization in the objective functional.

More in details, the main result is Theorem \ref{thconv1}: in each set $\dd^\rho$ (defined in Section \ref{COV}) $V$ is locally uniformly the limit of the functions $V^c$ in (\ref{Vc}). It follows that on $\bigcup_\rho \dd^\rho$ the value function is continuous (Corollary \ref{thconv2}). In Propositions \ref{PrRegVinc1} and \ref{PrRegVinc2} we prove that, under suitable assumptions, the function $V(t,\cdot,z)$ is Lipschitz continuous and a.e.~twice differentiable. In Section \ref{COVesempi} we show that Assumption \ref{assrho} is satisfied in the cases $g(s,v)=v$ and $g(s,v)=|v|^p$ ($p\geq1$), if $f$ and $\sigma$ satisfy an appropriate condition.

In Section \ref{CAPopzioniswing} we apply these general results to the problem in (\ref{intro2}). In this case stronger results will be achieved, since it can be proved directly that the value function is continuous not only in $\widetilde \dd \subsetneq \dd$ but in the whole domain $\dd$ (Proposition \ref{PrStrict1} and \ref{PrStrict2}). Thus, $V$ can be characterized as the unique viscosity solution with polynomial growth of the HJB equation under suitable boundary conditions (Theorem \ref{PrStrict3}). As for the regularity of the value function, besides the above cited general results about the variable $p$ (Proposition \ref{PrStrict4}), we prove that $V(t,p,\cdot)$ is concave and study its monotonicity (Proposition \ref{PrStrict5}).

Control problems with integral constraints are classical in control theory, for instance they naturally arise in applications: e.g.~control problems with bounded $L^p$ norm of the controls, control problems with prescribed bounded total variation or total energy of the trajectories, control systems with design uncertainties. However, the dynamic programming approach presents several technical difficulties. The main one relies on the fact that the dynamic programming principle is not satisfied directly by the value function and the problem has to be attacked differently. As for the case of deterministic systems, we refer to \cite{MS2,so} and references therein. As for the case of stochastic controls, the upper bound $Z(T) \leq M$ is analogous to the constraint of the so-called finite fuel problems, which are optimal control problems with an upper bound on the integral of the absolute value of the controls (see e.g.~\cite[Chapter VIII]{FS} for an introduction to the problem, \cite{MS} and references therein). Instead, the lower bound $Z(T) \geq m$ is nonstandard. In the particular case of Equation (\ref{intro2}), and only with $m = 0$, such a bound has been studied (treated in \cite{BLN} and generalized in \cite{BLN2}, still with $m = 0$). However, we already said that this case is quite unrealistic, as the seller wants to be sure to sell some amount of energy, so typically $m > 0$ (\cite{Loland} reports typical values of $m \in [0.8 M,M]$ for real contracts). To the best of our knowledge stochastic optimal control problems with general integral constraints on the controls have not yet been studied.

The structure of the paper is as follows. In Section \ref{CAPopzioniswingpenal} the evaluation problem for a swing contract with penalty is studied. Section \ref{CAPcontrottvinc} deals with a general class of constrained control problems, as in Equation (\ref{intro3}). Finally, in Section \ref{CAPopzioniswing} we deeply analyze the problem, outlined in (\ref{intro2}), of the optimal exercise of swing contracts with strict constraints.

\textit{Notations}. By $\| \cdot \|$ we denote the sup-norm. If $B \in M_{ij}(\rr)$ (i.e.~a real $i \times j$ matrix), $B^t$ denotes the transpose of $B$ and $\textup{tr}(B)$ denotes its trace. By $\overline{\text{B}(x,R)}$ we mean the closed ball in $\rr^n$ with center $x$ and radius $R$. If $O \subseteq \rr^n$ and $k \in \nn$, we denote by $C^k_b(O)$ (resp. $C^k_p(O)$) the set of functions of class $C^k(O)$ whose derivatives up to order $k$ are bounded (resp. are polynomially growing). If $\psi$ is a function from $(t,p,z) \in A \subseteq \rr \times \rr^n \times \rr$ to $\rr$, by $\psi_t,\psi_z$ we mean the derivatives with respect to $t$ and $z$ and by $D_p \psi, D_p^2 \psi$ we mean the Jacobian and the Hessian matrix with respect to the variable $p$.

\section{Swing contracts with penalties}
\label{CAPopzioniswingpenal}

In this section we consider the problem of the optimal exercise of swing contracts with penalties described in the Introduction: to this purpose, we formalize  a continuous time model to which we apply classical results in stochastic control.

Let $T>0$ and fix a filtered probability space ($\Omega$, $\mathcal{F}_T$, $\{\mathcal{F}_s\}_{s\in[0,T]}$, $\mathbb{P}$) and a real $\{\mathcal{F}_s\}_s$-adapted Brownian motion $W=\{W(s)\}_{s \in [0,T]}$. Let $t \in [0,T]$ and $p \geq 0$. We model the price of energy through a stochastic process $\{P^{t,p}(s)\}_{s \in [t,T]}$ which satisfies the SDE
\begin{equation}
\label{Pcasopartic}
dP^{t,p}(s)=f(s,P^{t,p}(s))ds + \sigma(s,P^{t,p}(s))dW(s),    \qquad s\in[t,T],
\end{equation}
with initial condition $P^{t,p}(t)=p$, where $f,\sigma \in C([0,T]\times \rr; \rr)$ are Lipschitz continuous with respect to the second variable uniformly in the first variable (it is well-known that this condition ensures the existence of a pathwise unique strong solution of the SDE).

In each $s\in[t,T]$, the holder can buy energy at a fixed unitary price $K>0$ and with purchase intensity $u(s) \in[0,\bar u]$, where $\bar u>0$ is a constant: this gives a net instantaneous profit (or loss) of $(P^{t,p}(s)-K)u(s)$. Let $\aaa$ be the set of all $[0,\bar u]$-valued progressively measurable processes $u=\{u(s)\}_{s \in [t,T]}$ (i.e.~all the possible usage strategies of the contract). Let $z$ be the amount of energy purchased until time $t$ and let $u \in \aaa$ be an exercise strategy from time $t$ on; for each $s \in [t,T]$ we denote by $Z^{t,z;u}(s)$ the energy bought up to time $s$:
\begin{equation*}
Z^{t,z;u}(s) =z+\int_t^s u(\tau)d\tau, \qquad s \in [t,T].
\end{equation*}

If the globally purchased energy $Z^{t,z;u}(T)$ does not fall within a fixed range $[m,M]$ ($m,M \in \rr$, with $m \leq M$), the holder must pay a penalty $\widetilde\Phi(P^{t,p}(T),Z^{t,z;u}(T))$, where $\widetilde \Phi$ is a function from $\rr^2$ to $\rr$. In the typical case (see, for example, \cite{BBP,Barrera,EV}) the penalty is directly proportional to $P^{t,p}(T)^+$ and to the entity of the overrunning or underrunning: this is obtained by setting
\begin{equation*}
\widetilde\Phi(p,z) = - A p^+ (z-M)^+ - B p^+ (m-z)^+,
\end{equation*}
for all $(p,z)\in \rr^2$, where $A, B >0$ are suitable constants. In several practical cases, $A = B$. However, other kind of penalties are possible (see e.g.~\cite{JRT}): typically $p^+$, representing the spot price at the end $T$ of the contract, is replaced either by an arithmetic mean of spot prices (thus requiring another state variable in the problem) or by a fixed (high) penalty. In the light of the above discussion, we assume that $\widetilde \Phi$ is null for $z \in [m,M]$, globally concave in $z$ and such that
$$ | \widetilde\Phi(p + h,z) - \widetilde\Phi(p,z) | \leq C h (1+|z|), \,\,\,  | \widetilde\Phi(p,z + h) - \widetilde\Phi(p,z) | \leq C h (1+|p|), \,\,\, \forall (p,z) \in \rr^2, h>0, $$
where $C>0$ is a constant.

Let $r\geq 0$ be the risk-free rate. We get a stochastic optimal control problem, with the following value function:
\begin{equation}
\label{penalita}
\widetilde V(t,p,z)= \sup_{u \in \aaa} \widetilde J(t,p,z;u),
\end{equation}
for each $(t,p,z)\in [0,T]\times \rr^2$, where
\begin{equation*}
\widetilde J(t,p,z;u)=\eee\biggl[\int_t^T e^{-r(s-t)}(P^{t,p}(s)-K)u(s) ds + e^{-r(T-t)}\widetilde\Phi(P^{t,p}(T),Z^{t,z;u}(T))\biggr]
\end{equation*}
and by $\eee$ we denote the mean value with respect to the probability $\mathbb{P}$ (subscripts recall initial conditions).

Problem (\ref{penalita}) belongs to a widely studied class of control problems: by well-known classical results, summarized in Theorem \ref{PrPen1}, the value function is the unique viscosity solution of the corresponding Hamilton-Jacobi-Bellman equation subject to appropriate conditions (we refer to \cite{USER} for the definition of viscosity solutions).

\begin{thm}
\label{PrPen1}
The function $\widetilde V$ is the unique viscosity solution of
\begin{multline}
\label{HJBpenalt}
-\widetilde V_t(t,p,z) + r \widetilde V(t,p,z) - f(t,p) \widetilde V_p(t,p,z) - \frac{1}{2}\sigma^2(t,p) \widetilde V_{pp}(t,p,z) \\
+ \min_{v \in [0,\bar u]}[ -v (\widetilde V_z(t,p,z) + p-K) ]=0, \qquad \forall (t,p,z) \in [0,T[\times\rr^2,
\end{multline}
with final condition
\begin{equation}
\label{HJB2penalt}
\widetilde V(T,p,z)=\widetilde\Phi(p,z), \qquad \forall (p,z) \in \rr^{2},
\end{equation}
and such that
\begin{equation*}
|\widetilde V(t,p,z)| \leq \check C (1+|p|^2+|z|^2), \qquad \forall (t,p,z) \in [0,T] \times \rr^2,
\end{equation*}
for some constant $\check C >0$.
\end{thm}

\begin{proof}
See Theorem \ref{propVc}, of which this theorem is a particular case.
\end{proof}

\begin{remark}
Since the solution of problem (\ref{HJBpenalt})-(\ref{HJB2penalt}) is unique, the value function $\widetilde V$ does not depend on the probability space chosen.
\end{remark}

We now list some properties of the function $\widetilde V$ with respect to the variables $p$ and $z$. Let us start by proving regularity results with respect to the variable $p$. %that the function $\widetilde V(t,\cdot,z)$ is Lipschitz continuous.

\begin{prop}
\label{PrPen2}
For each $(t,z) \in [0,T] \times \rr$ the function $\widetilde V(t,\cdot, z)$ is Lipschitz continuous, uniformly in $t$. Moreover, the derivative $\widetilde V_p(t,p,z)$ exists for a.e.~$(t,p,z) \in [0,T] \times {\rr^2}$ and we have $|\widetilde V_p(t,p,z)|\leq M_1(1+|z|)$, for some constant $M_1>0$ depending only on $\bar u$, $T$, $C$ and on the constants in (\ref{fs}) and (\ref{cpolfs}).
\end{prop}

\begin{proof}
Let $(t,p,z) \in [0,T]\times \rr^2$, $h>0$ and $u \in \aaa$. By estimate (D.8) in \cite[Appendix D]{FS} we have
\begin{align}
\label{stima1}
& |\widetilde J(t,p+h,z;u)- \widetilde J(t,p,z;u)| \nonumber \\
& \leq \eee\left[\int_t^T |P^{t,p+h}(s)-P^{t,p}(s)| \, |u(s)|ds +C|P^{t,p+h}(T)-P^{t,p}(T)| \, (1+|Z^{t,z;u}(T)|)\right] \nonumber \\
& \leq T \bar u \, \eee [\|P^{t,p+h}(\cdot)-P^{t,p}(\cdot)\|] + C\eee [\|P^{t,p+h}(\cdot)-P^{t,p}(\cdot)\|] (1+|z|+\bar u T) \nonumber \\
& \leq M_1(1+|z|)h,
\end{align}
where $M_1>0$ is a constant. Since (\ref{stima1}) holds for each $u \in \aaa$, we get
\begin{equation}
\label{stima1bis}
|\widetilde V(t,p+h,z)-\widetilde V(t,p,z)| \leq M_1 (1+|z|) h.
\end{equation}
The function $\widetilde V(t,\cdot, z)$ is therefore Lipschitz continuous, uniformly in $t$, and then a.e.~differentiable by the Rademacher theorem. By standard arguments it follows that $\widetilde V_p(t,p,z)$ exists for a.e.~$(t,p,z) \in [0,T] \times {\rr^2}$. The estimate on the derivative immediately follows by (\ref{stima1bis}).
\end{proof}

In the following proposition we collect some results about smoothness and monotonicity of the function $\widetilde V$ with respect to $z$. %$\widetilde V(t,p,\cdot)$.

\begin{prop}
\label{PrPen3}
For each $(t,p) \in [0,T]\times \rr$ the function $\widetilde V(t,p,\cdot)$ is
\begin{itemize}
\item[-] Lipschitz continuous, uniformly in $t$. Moreover, the derivative $\widetilde V_z(t,p,z)$ exists for a.e. $(t,p,z) \in [0,T] \times {\rr^2}$ and we have $|\widetilde V_z(t,p,z)|\leq M_2(1+|p|)$, for some constant $M_2>0$ depending only on $\bar u$, $T$, $C$ and on the constants in (\ref{fs}) and (\ref{cpolfs}).
\item[-] concave and a.e.~twice differentiable;
\item[-] weakly increasing in $]-\infty, M-(T-t)\bar u]$ and weakly decreasing in $[m,+\infty[$. In particular, if $M-(T-t)\bar u \geq m$ then the function $ \widetilde V(t,p,\cdot)$ is constant in $[m, M-(T-t)\bar u]$ (they all are maximum points).
\end{itemize}
\end{prop}

\begin{proof}
\textit{Item 1.} Let $(t,p,z) \in [0,T]\times \rr^2$, $h>0$ and $u \in \aaa$.
Recall the following estimate from \cite[Appendix D]{FS}: for each $k \geq 0$ there exists a constant $B_k \geq0$, depending only on $\bar u$, $T$, $C$ and on the constants in (\ref{fs}) and (\ref{cpolfs}), such that
\begin{equation}
\label{momenti}
\eee\big[{ \|P^{t,p;u}(\cdot)\| }^k \big] \leq B_k (1+|p|^k).
\end{equation}
By this and the Lipschitzianity of $\widetilde\Phi(P^{t,p}(T),\cdot)$ %and by (\ref{momenti})
we have
\begin{multline*}
|\widetilde J(t,p,z+h;u)- \widetilde J(t,p,z;u)| \leq C \eee[(1+|P^{t,p}(T)|) \, |Z^{t,z+h;u}(T)-Z^{t,z;u}(T)|] \\
\leq C h (1+\eee[|P^{t,p}(T)|]) \leq M_2 (1+|p|) h,
\end{multline*}
where $M_2>0$ is a constant. Then argue as in Proposition \ref{PrPen2}.

\textit{Item 2.} Let $(t,p) \in [0,T]\times \rr$, $z_1,z_2 \in \rr$ and $u_1,u_2 \in \aaa$. Notice that $(u_1+u_2)/2 \in \aaa$ and that
\begin{equation}
\label{stima2}
Z^{t,\frac{z_1+z_2}{2};\frac{u_1+u_2}{2}}(T)= \frac{Z^{t,z_1;u_1}(T)+Z^{t,z_2;u_2}(T)}{2}.
\end{equation}
By the concavity of the function $\widetilde \Phi (P^{t,p}(T), \cdot)$ and by (\ref{stima2}) we have
\begin{equation}
\label{stima3}
\frac{\widetilde J(t,p,z_1;u_1)+ \widetilde J(t,p,z_2;u_2)}{2} \leq  \widetilde J \left(t,p,\frac{z_1+z_2}{2}; \frac{u_1+u_2}{2} \right) \leq V\left(t,p,\frac{z_1+z_2}{2}\right).
\end{equation}
Since (\ref{stima3}) holds for each $u_1, u_2 \in \aaa$, we get
\begin{equation*}
\frac{\widetilde V(t,p,z_1)+ \widetilde V(t,p,z_2)}{2} \leq \widetilde V \left(t,p,\frac{z_1+z_2}{2}\right),
\end{equation*}
and then the concavity of the function $\widetilde V(t,p,\cdot)$. The a.e.~existence of the second derivative follows from the Alexandrov theorem.

\textit{Item 3.} Let $(t,p) \in [0,T] \times \rr$, $z_1 \leq z_2 \leq M - (T-t)\bar u$ (the case $m \leq z_1 \leq z_2$ is similar) and $u \in \aaa$. Since
\begin{equation*}
Z^{t,z_1;u}(T) \leq Z^{t,z_2;u}(T) = z_2 + \int_t^T u(s) ds \leq z_2 + (T-t)\bar u \leq M
\end{equation*}
and since the function $\widetilde \Phi (P^{t,p}(T), \cdot)$ is weakly increasing in $]-\infty,M]$, we have that
\begin{equation}
\label{stima4}
\widetilde J (t,p,z_1;u) \leq \widetilde J (t,p,z_2;u).
\end{equation}
As inequality (\ref{stima4}) holds for each $u \in \aaa$, we get
\begin{equation*}
\widetilde V(t,p,z_1) \leq \widetilde V(t,p,z_2).
\end{equation*}
The second part immediately follows, since $]-\infty, M-(T-t)\bar u] \cap [m, +\infty[=[m, M-(T-t)\bar u]$.
\end{proof}

The monotonicity result in Proposition \ref{PrPen3} is described in Figure \ref{PICcrescpenalty}.
\begin{figure}[htb]
\centering
\includegraphics[width=0.25\columnwidth]{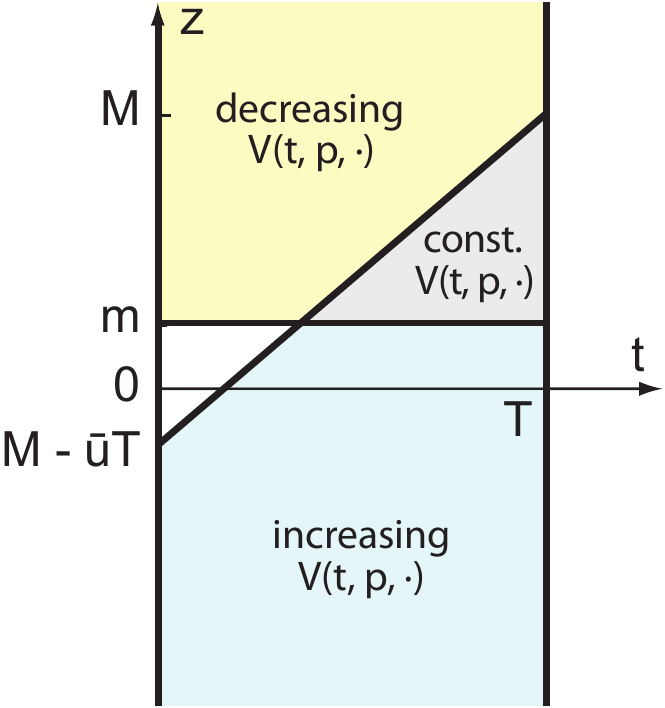}
\caption{monotonicity of $\widetilde V(t,p,\cdot)$}
\label{PICcrescpenalty}
\end{figure}

The third part of Proposition \ref{PrPen3} provides an apparently unexpected result: for suitable $t$ and for all $p$, the function $V(t,p,\cdot)$ is constant in an interval. As a matter of fact, this was foreseeable: it is easy to check that if $M-(T-t)\bar u \geq m$ and $z \in [m, M-(T-t)\bar u]$ then $Z^{t,z;u}(T) \in [m,M]$ for each $u \in \aaa$, so that the penalization term in the objective functional vanishes and the initial value $z$ does not influence the value function.

\begin{remark}
\label{remark}
As observed in \cite[Equation (3.9)]{BLN}, by (\ref{HJBpenalt}) a candidate optimal control policy is
\begin{equation}
\label{candidate}
\underline{u}(t,p,z)=
\begin{cases}
\bar u \qquad  \text{if $\widetilde V_z(t,p,z)\geq p-K$},\\
0 \qquad  \text{if $\widetilde V_z(t,p,z) < p-K$}.
\end{cases}
\end{equation}
Notice that by Proposition \ref{PrPen3} the candidate in (\ref{candidate}) is a.e.~well-defined. Moreover, since $\widetilde V$ is concave in $z$, for each fixed $(t,p)$ there exists $\bar z(t,p) \in [-\infty,+\infty]$ such that $\widetilde V_z(t,p,z) < p - K$ if and only if $z > \bar z(t,p)$: for $t$ fixed, the function $\bar z(t,\cdot)$ (which in \cite{BLN} is called exercise curve) can be used to write $\underline u$ as
\begin{equation}
\label{candidate2}
\underline{u}(t,p,z) =
\begin{cases}
\bar u \qquad  \text{if } z \leq \bar z(t,p),\\
0 \qquad  \text{if } z > \bar z(t,p).
\end{cases}
\end{equation}
\end{remark}

\section{Integral constrained stochastic optimal control}
\label{CAPcontrottvinc}

Let us now consider the problem, outlined in the Introduction, of optimally exercising swing contracts with strict constraints. Due to the presence of the constraint, in this case it is not possible to argue as in Section \ref{CAPopzioniswingpenal} and use classical results in control theory. This motivates us to study a more general class of stochastic optimal control problems with integral constraints, of which swing contracts with strict constraints will be a particular case.

\subsection{Formulation of the problem}
\label{COV}

Let $d,l,n \in \mathbb{N}$, $r \geq 0$, $T>0$ and $m,M\in\rr$ with $m<M$. Let $U \subseteq \rr^l$ be nonempty and $f, \sigma, g, L, \Phi$ be functions satisfying the following assumptions:

\begin{ipo}
\begin{itemize}
\item[i)] $U$ is a compact subset of $\rr^l$;
\item[ii)]$f \in C([0,T]\times \rr^n \times U ; \rr^n)$, $\sigma \in C([0,T]\times \rr^n \times U ; M_{nd}(\mathbb{R}))$ and there exists a constant $C>0$ such that
\begin{equation}
\label{fs}
\begin{split}
& |f(t,p,v)-f(t,q,v)| \leq C|p-q|, \quad \forall p,q \in \rr^n, \,\,\, \forall (t,v)\in[0,T]\times U, \\
& |\sigma(t,p,v)-\sigma(t,q,v)| \leq C|p-q|, \quad \forall p,q \in \rr^n, \,\,\, \forall (t,v)\in[0,T]\times U;
\end{split}
\end{equation}
\item[iii)] $g \in C([0,T]\times U ; \rr)$;
\item[iv)] $L \in C([0,T] \times \rr^{n} \times \rr \times U ; \rr)$, $\Phi \in C(\rr^{n} \times \rr ; \rr)$ and there exist constants $\tilde C,k>1$ such that
\begin{equation}
\label{lphiV}
\begin{split}
& |L(t,p,z,v)| \leq \tilde C(1 + |p|^k + |z|^k), \,\,\quad \forall (t,p,z,v)\in [0,T] \times \rr^{n} \times \rr \times U,\\
& |\Phi(p,z)| \leq \tilde C(1 + |p|^k + |z|^k), \,\,\quad \forall (p,z)\in \rr^{n} \times \rr.
\end{split}
\end{equation}
    Moreover, for each compact subset $A \subseteq \rr^{n+1}$ there exists a modulus of continuity $\omega_A$ such that
\begin{equation}
|L(t,p,z,v)-L(t,q,y,v)|\leq \omega_R(|p-q|+|z-y|),
\end{equation}
    for all $(t,v)\in[0,T]\times U$ and for all $(p,z),(q,y)\in A$.
\end{itemize}
\end{ipo}
\noindent Notice that conditions (\ref{fs}) implies that
\begin{equation}
\label{cpolfs}
\begin{split}
& |f(t,p,v)| \leq \hat C(1 + |p|), \qquad \forall (t,p,v)\in [0,T] \times \rr^n \times U,\\
& |\sigma(t,p,v)| \leq \hat C(1 + |p|), \qquad \forall (t,p,v)\in [0,T] \times \rr^n \times U,\\
\end{split}
\end{equation}
where $\hat C >0$ is a constant.

Let $(\Omega, \mathcal{F}_T, \{\mathcal{F}_s\}_{s\in[0,T]}, \mathbb{P}, W)$ be a fixed filtered probability space where a $d$-dimensional $\{\mathcal{F}_s\}_s$-adapted Brownian motion $W=\{W(s)\}_{s \in [0,T]}$ is defined. If $t \in [0,T]$, let $\aaa$ denote the set of all $U$-valued progressively measurable processes $u=\{u(s)\}_{s\in[t,T]}$ (\emph{controls}) such that for each $p \in \rr^n$ the $n$-dimensional stochastic differential equation.
\begin{equation}
\label{SDEV}
dP^{t,p;u}(s)=f(s,P^{t,p;u}(s),u(s))ds + \sigma(s,P^{t,p;u}(s),u(s))dW(s),    \qquad s\in[t,T],
\end{equation}
with initial condition
\begin{equation}
\label{SDE2V}
P^{t,p;u}(t)=p,
\end{equation}
has a pathwise unique strong solution.

Let $t \in [0,T]$, $z \in \rr$, $u \in \aaa$ and let
\begin{equation}
\label{Z}
Z^{t,z;u}(s)= z+ \int_t^s g(\tau,u(\tau))d\tau,  \qquad s\in[t,T].
\end{equation}
A control $u \in \aaa$ is called \emph{admissible} if the process $Z^{t,z;u}$ a.s. reaches the interval $[m,M]$ at the final time $T$:
\begin{equation*}
\aav= \left \{ \text{$u \in \aaa$ : $Z^{t,z;u}(T) \in [m,M]$ \,\, $\ppp$-a.s.} \right \}.
\end{equation*}
We will often write $P^u$ and $Z^u$, in order to shorten the notations.

Given $(t,z) \in [0,T] \times \rr$ and $A \subseteq \rr$, we say that \emph{$A$ is reachable from $(t,z)$} if there exists a Borel measurable function $u$ from $[t,T]$ to $U$ (notice that then $u \in \aaa$) such that $Z^{t,z;u}(T)\in A$. Let $\dd, \widetilde \dd, \dd^\rho$ (for $0<\rho<(M-m)/2$) denote the subsets of $[0,T] \times \rr^{n} \times \rr$ defined by
\begin{gather*}
\dd= \left \{ \text{$(t,p,z) \in [0,T] \times \rr^{n} \times \rr$ : $[m,M]$ is reachable from $(t,z)$} \right \},\\
\widetilde \dd= \left \{ \text{$(t,p,z) \in [0,T] \times \rr^{n} \times \rr$ : $]m,M[$ is reachable from $(t,z)$} \right \},\\
\dd^ \rho= \left \{ \text{$(t,p,z) \in [0,T] \times \rr^{n} \times \rr$ : $[m+\rho,M-\rho]$ is reachable from $(t,z)$} \right \}.
\end{gather*}
Notice that $\bigcup_\rho \dd^\rho = \widetilde \dd \subseteq \dd$. It is easy to prove that these sets are nonempty.

\begin{lemma}
The sets $\dd, \widetilde \dd, \dd^\rho$ are nonempty.
\end{lemma}

\begin{proof}
Let $0 < \rho < (M-m)/2$. As $\dd^\rho \subseteq \widetilde \dd \subseteq \dd$, it suffices to show that $\dd^\rho \neq \varnothing$. Since $g([0,T]\times U)=[\xi_1,\xi_2]$, for suitable $(\tilde t,\tilde z) \in [0,T]\times\rr$ we have that $Z^{\tilde t,\tilde z,u}(T)\in[\tilde z+\xi_1(T-\tilde t),\tilde z+\xi_2(T-\tilde t)]\subseteq[m+\rho,M-\rho]$ for each Borel measurable function $u$ from $[\tilde t,T]$ to $U$, and thus $(\tilde t, \tilde p, \tilde z) \in \dd^\rho$ (arbitrary $\tilde p \in \rr$).
\end{proof}

If $(t,p,z)\in\dd$, by $\eee$ we denote the mean value with respect to the probability $\ppp=\mathbb{P}$ (subscripts recall initial data). We can now define the \emph{value function}.

\begin{defn}
We set
\begin{equation}
\label{VV}
V(t,p,z)=\sup_{u \in \aav} J(t,p,z;u),
\end{equation}
for each $(t,p,z) \in \dd$, where
\begin{multline*}
J(t,p,z;u) = \eee \biggl[ \int_t^T e^{-r(s-t)}L(s,P^{t,p;u}(s),Z^{t,z;u}(s),u(s))ds \\
+ e^{-r(T-t)}\Phi(P^{t,p;u}(T),Z^{t,z;u}(T))\biggr].
\end{multline*}
\end{defn}

Let us prove that the value function (\ref{VV}) is well defined.

\begin{lemma}
The expectations in (\ref{VV}) are well posed, $V(t,p,z)<\infty$ and
\begin{equation}
\label{VVpol}
|V(t,p,z)| \leq \Gamma (1+|p|^k+|z|^k),
\end{equation}
for each $(t,p,z)\in \dd$, where $k$ is as in (\ref{lphiV}) and $\Gamma \geq 0$ is a constant depending only on $U$, $T$, $C$, $\tilde C$, $\hat C$ and $\max g$. Moreover, $\dd$ is the maximal set in which expression (\ref{VV}) makes sense.
\end{lemma}

\begin{proof}
First of all, notice that $\aav \neq \varnothing$ if and only if $(t,p,z) \in \dd$ for each $p \in \rr^n$.
%Recall now the following estimate from \cite[Appendix D]{FS}: for each $k \geq 0$ there exists a constant $B_k \geq0$ such that
%\begin{equation}
%\label{momenti}
%\eee\big[{ \|P^{t,p;u}(\cdot)\| }^k \big] \leq B_k (1+|p|^k).
%\end{equation}
Recall estimate (\ref{momenti}): it can be shown that $B_k$ depends only on the set $U$ and on constants $T$, $C$, $\hat C$ (see \cite[Appendix D]{FS}. By (\ref{lphiV}) and (\ref{momenti}) we have
\begin{align}
\label{Vpol}
& \eee \bigg[ \, \bigg|\int_t^T e^{-r(s-t)}L(s,P^{t,p;u}(s),Z^{t,z;u}(s),u(s))ds + e^{-r(T-t)}\Phi(P^{t,p;u}(T),Z^{t,z;u}(T))\bigg| \, \bigg] \nonumber \\
& \leq \tilde C \eee \bigg[ \int_t^T (1+|P^{t,p;u}(s)|^k+|Z^{t,z;u}(s)|^k)ds + (1+|P^{t,p;u}(T)|^k+|Z^{t,z;u}(T)|^k)\bigg] \nonumber \\
& \leq C_1 \eee [ 1+\|P^{t,p;u}(\cdot)\|^k+\|Z^{t,z;u}(\cdot)\|^k] \nonumber \\
& \leq C_2 (1+|p|^k+|z|^k),
\end{align}
for suitable constants $C_1,C_2>0$.
\end{proof}

We also require the following assumption to hold.

\begin{ipo}
\label{assrho}
Given $0<\rho<(M-m)/2$ and a compact subset $A \subseteq \dd^\rho$, there exist $\bar \eps>0$ and a function $0< \eta \leq (M-m)/2$, both depending only on $\rho$, $A$, $T$ and $U$, with the following property: for each $0<\eps<\bar\eps$, $(t,p,z)\in A$, $u \in\aav$ there exists $\tilde u \in \aaa$ such that $|J(t,p,z;u)-J(t,p,z;\tilde u)|\leq \eps$ and that a.s.~$Z^{t,z;\tilde u}(T)\in [m+\eta(\eps),M-\eta(\eps)]$.
\end{ipo}

In Section \ref{COVesempi} we will give two examples of wide classes of problems satisfying Assumption \ref{assrho}.

\subsection{Approximating problems}

We would like to obtain for the problems of Section \ref{COV} the standard results in unconstrained control theory: continuity of the value function and characterization of the value function by the Hamilton-Jacobi-Bellman (HJB) equation. A straightforward approach is not possible, since condition $u \in \aav$ prevents from simply adapting classical proofs.

The idea is then the following: to define suitable unconstrained problems which approximate our constrained problem and then to obtain the properties of the value function (\ref{VV}) through a limiting procedure. The construction of the approximating problems is based on the idea of penalizing the case $Z^{t,z;u}(T) \notin [m,M]$ by adding a suitable term in the objective functional.

Given $c>0$, let $\Phi^c$ be the function from $\rr$ to $\rr$ defined by
\begin{equation}
\label{phic}
\Phi^c(z) = - c \left[\left(z-\left(M-\frac{1}{\sqrt{c}}\right)\right)^+ + \left(\left(m+\frac{1}{\sqrt{c}}\right)-z\right)^+ \right],
\end{equation}
for each $z \in \rr$. Let the assumptions of Section \ref{COV} hold and consider the following unconstrained problem:
\begin{equation}
\label{Vc}
V^c(t,p,z)= \sup_{u \in \aaa}  J^c(t,p,z;u),
\end{equation}
where $(t,p,z) \in [0,T]\times \rr^{n} \times \rr$,
\begin{multline*}
J^c(t,p,z;u) = \eee \biggl[ \int_t^T e^{-r(s-t)}L(s,P^{t,p;u}(s),Z^{t,z;u}(s),u(s)) ds \\
+ e^{-r(T-t)}\Phi(P^{t,p;u}(T),Z^{t,z;u}(T)) + e^{-r(T-t)}\Phi^c(Z^{t,z;u}(T))\biggr]
\end{multline*}
and this time the maximization is performed over the set $\aaa$ of all controls.

Problem (\ref{Vc}) is a classical unconstrained stochastic control problem; therefore, by classical results, the function $V^c$ is characterized by the HJB equation. Here is the precise statement.

\begin{thm}
\label{propVc}
Let the assumptions of Section \ref{COV} hold and let $c>0$ and $k$ as in (\ref{lphiV}). Then $V^c$ is the unique viscosity solution of
\begin{multline}
\label{HJBc}
-V^c_t(t,p,z) + r V^c(t,p,z) + \min_{v \in U}\biggl[- f(t,p,v) \cdot D_pV^c(t,p,z) - g(t,v) V^c_z(t,p,z) \\
-\frac{1}{2}\textup{tr} \big(\sigma(t,p,v) \sigma^t(t,p,v) D^2_p V^c(t,p,z)\big) - L(t,p,z,v) \biggr]=0, \,\,\, \forall (t,p,z) \in [0,T[ \times \rr^{n} \times \rr,
\end{multline}
with final condition
\begin{equation}
\label{HJB2c}
V^c(T,p,z)=\Phi(p,z)+\Phi^c(z), \qquad \forall (p,z) \in \rr^{n+1},
\end{equation}
and such that
\begin{equation*}
|V^c(t,p,z)| \leq \check C (1+|p|^{k}+|z|^{k}), \qquad \forall (t,p,z) \in [0,T] \times \rr^n \times \rr,
\end{equation*}
for some constant $\check C >0$.
\end{thm}

\begin{proof}
The value function is a viscosity solution of (\ref{HJBc}) by a standard result in unconstrained control theory (the proof, for instance, can be achieved by slightly modifying the arguments in \cite[Chapter IV]{FS}). As for uniqueness, see \cite[Thm.~3.1]{DaLioLey2}.
\end{proof}

\subsection{Properties of the value function}
\label{COVpropfv}

We now prove the central result of this paper: the value functions $V^c$ in (\ref{Vc}) converge, uniformly on the compact subsets of each $\dd^\rho$, to the value function $V$ in (\ref{VV}). On one hand, this results provides an approximation of $V$ (recall the characterization of the functions $V^c$ in Theorem \ref{propVc}); on the other hand, in such a way $V$ inherits continuity from the functions $V^c$.

\begin{thm}
\label{thconv1}
Let the assumptions of Section \ref{COV} hold. Then, as $c \to +\infty$, the functions $V^c$ converge to $V$ uniformly on compact subsets of $\dd^\rho$, for each $0<\rho<(M-m)/2$.
\end{thm}

\begin{proof}
Let $0<\rho<(M-m)/2$, $A$ be a compact subset of $\dd^\rho$ and $R>0$ be such that $\overline{\text{B}(0,R)}\supseteq A$. For each $\eps >0$, we have to prove that there exists $\delta>0$ such that
\begin{equation}
\label{Tconverg1}
\bigg |\sup_{u \in \aaa} J^c(t,p,z;u)- \sup_{u \in \aav} J(t,p,z;u) \bigg| \leq \eps,
\end{equation}
for each $c\geq\delta$ and $(t,p,z)\in A$.

\textit{Step 1: lower bound for $V^c$ in $A$}. By definition of $\dd^\rho$, for each $(t,p,z)\in A$ let $u_{tpz}$ be a Borel measurable function from $[t,T]$ to $U$ such that $Z^{u_{tpz}}(T)\in[m+\rho,M-\rho]$. Since $[m+\rho,M-\rho] \subseteq [m+ c ^{-\frac{1}{2}},M- c ^{-\frac{1}{2}}]$ for $c \geq \rho^{-2}$, notice that $J^c(t,p,z;u_{tpz}) \equiv K_{tpz}$ for a suitable constant $K_{tpz}$ for all  $c \geq  \rho^{-2}$. By estimates as in (\ref{Vpol}), it is easy to show that, for a constant $C_1>0$ and for $k \geq 0$ as in (\ref{lphiV}), we have $|K_{tpz}|\leq C_1(1+|p|^k+|z|^k) \leq C_1(1+2 R^k)$ for each $(t,p,z)\in A$, so that $K:= \inf_{(t,p,z)\in A}K_{tpz}\in\rr$. Therefore,
\begin{equation}
\label{Tconverg2}
V^c(t,p,z) \geq J^c(t,p,z;u_{tpz})=K_{tpz}\geq K,
\end{equation}
for each $(t,p,z) \in A$ and $c\geq  \rho^{-2}$.

\textit{Step 2:  new formulation of (\ref{Tconverg1})}. Let $(t,p,z)\in A$. For each $n \in \nn$ we set
\begin{equation*}
B_n^{tpz}=\left \{ u \in \aaa : \frac{1}{n+1}< \ppp(Z^u(T) \notin [m,M]) \leq \frac{1}{n}  \right \}.
\end{equation*}
Let $c \geq \rho^{-2}$, $n \in \nn$ and $u \in B_n^{tpz}$. By noting that $\Phi^c\leq0$ and that $\Phi^c (x)\leq -\sqrt{c}$ for $x \notin [m,M]$ and by estimates as in (\ref{Vpol}), we have
\begin{align}
\label{Tconverg3}
J^c(t,p,z;u) & = J(t,p,z;u) + \eee\left[e^{-r(T-t)}\Phi^c(Z^u(T))\right] \nonumber \\
& \leq J(t,p,z;u) + \eee\left[e^{-r(T-t)}\Phi^c(Z^u(T)) \mathbbm{1}_{ \{Z^u(T)\notin[m,M]\} }\right]  \nonumber \\
& \leq C_2(1+|p|^k+|z|^k) - e^{-rT}\sqrt{c} \,\, \ppp( Z^u(T)\notin[m,M] )  \nonumber \\
& < C_2(1+2R^k) - e^{-rT}\frac{ \sqrt{c} }{n+1},
\end{align}
for a suitable constant $C_2>0$. By (\ref{Tconverg3}) it follows that for each $n \in \nn$ there exists $c(n)\geq\rho^{-2}$ such that
\begin{equation*}
J^c(t,p,z;u) < K,
\end{equation*}
for each $c \geq c(n)$ and $u \in B_n^{tpz}$, with $K$ as in Step 1. By (\ref{Tconverg2}) we thus get
\begin{equation}
\label{Tconverg4}
\sup_{u \in \aaa} J^c(t,p,z;u) =  \sup_{u \in \aaa \setminus \bigcup_{ \{ i \in \nn : c(i)\leq c \} } B_i^{tpz} } J^c(t,p,z;u),
\end{equation}
for each $c \geq \rho^{-2}$. The sequence $\{c(n)\}_n$ is obviously increasing; hence, there exists a function $m$ from $[\rho^{-2}, +\infty[$ to $\nn$ such that $\{ i \in \nn : c(i)\leq c \} = \{1,\dots,m(c)\}$ for each $c\geq\rho^{-2}$. As a consequence, we can rewrite (\ref{Tconverg4}) as follows:
\begin{equation}
\label{Tconverg5}
\sup_{u \in \aaa} J^c(t,p,z;u) =  \sup_{u \in \aaa \setminus \bigcup_{i=1}^{m(c)}B_i^{tpz} } J^c(t,p,z;u).
\end{equation}
Notice that $m(\cdot)$ is increasing and that $m(c)\to +\infty$.

Let $\eps >0$. By (\ref{Tconverg5}), for $c \geq \rho^{-2}$ inequality (\ref{Tconverg1}) is equivalent to
\begin{equation}
\label{Tconverg6}
-\eps \leq \sup_{u \in \aaa \setminus \bigcup_{i=1}^{m(c)}B_i^{tpz} } J^c(t,p,z;u) - \sup_{u \in \aav } J(t,p,z;u) \leq \eps.
\end{equation}
Therefore, we have to prove that there exists $\delta\geq\rho^{-2}$ such that (\ref{Tconverg6}) holds for each $c\geq\delta$ and for each $(t,p,z)\in A$. In Step 3 we will prove the right inequality in (\ref{Tconverg6}), while in Step 4 the left inequality will be proved, thus concluding the proof.

\textit{Step 3: right inequality in (\ref{Tconverg6})}. Let us show that there exists $\delta_1\geq\rho^{-2}$ independent of $(t,p,z)\in A$ such that
\begin{equation}
\label{Tconverg7}
\sup_{u \in \aaa \setminus \bigcup_{i=1}^{m(c)}B_i^{tpz} } J(t,p,z;u) \leq \sup_{u \in \aav } J(t,p,z;u) + \eps,
\end{equation}
for each $c\geq\delta_1$. Since $J^c \leq J$, by (\ref{Tconverg7}) we get the right inequality in (\ref{Tconverg6}).

Let $c\geq\rho^{-2}$, $(t,p,z)\in A$, $u \in \aaa \setminus \bigcup_{i=1}^{m(c)}B_i^{tpz}$. We set
\begin{equation*}
\Pi^{u} = \{Z^u(T) \notin [m,M] \};
\end{equation*}
notice that $0\leq \ppp(\Pi^{u})\leq 1/(m(c)+1)$. Let $\tilde u$ be the process defined in the following way: $\tilde u$ coincides in $\Pi^{u}$ with the process which assures the reachability of $[m+\rho,M-\rho]$ (see the definition of $\dd^\rho$), and $\tilde u\equiv u$ in $\Omega \setminus \Pi^u$. A simple check shows that $\tilde u \in \aaa$ and that
\begin{equation}
\label{Tconverg8}
Z^{\tilde u}(T) \in [m,M] \qquad \ppp\text{-a.s.}
\end{equation}
By recalling that $\tilde u \equiv u$ in $\Omega \setminus \Pi^u$, by the H\"{o}lder inequality (twice) and by estimates as in (\ref{Vpol}), we obtain that
\begin{equation}
\label{Tconverg9}
\big|J(t,p,z;u) - J(t,p,z; \tilde u)\big| \leq C_3(1+|p|^k+|z|^k) \, \ppp(\Pi^{u})^{\frac{1}{2}} \leq\frac{ C_3(1+2R^k)}{(m(c)+1)^{\frac{1}{2}}},
\end{equation}
for some constant $C_3>0$. Then, by (\ref{Tconverg9}) and (\ref{Tconverg8}) it follows that
\begin{equation*}
J(t,p,z;u) \leq J(t,p,z; \tilde u)+ \frac{C_3(1+2R^k)}{(m(c)+1)^{\frac{1}{2}}} \leq \sup_{u \in\aav} J(t,p,z; u)+ \frac{C_3(1+2R^k)}{(m(c)+1)^{\frac{1}{2}}}.
\end{equation*}
This inequality holds for each $(t,p,z)\in A$ and $u \in \aaa \setminus\bigcup_{i=1}^{m(c)}B_i^{tpz}$. Since $m(c) \to +\infty$, for sufficiently large $c$ (and this choice is independent of $(t,p,z)$ and $u$), we have that $C_3(1+2R^k)/(m(c)+1)^{\frac{1}{2}} \leq \eps$, thus obtaining (\ref{Tconverg7}).

\textit{Step 4: left inequality in (\ref{Tconverg6})}. We still have to prove the left inequality in (\ref{Tconverg6}), i.e.
\begin{equation}
\label{Tconverg10}
\sup_{u \in \aav } J(t,p,z;u) \leq \sup_{u \in \aaa \setminus \bigcup_{i=1}^{m(c)}B_i^{tpz} } J^c(t,p,z;u) + \eps,
\end{equation}
for $c\geq\delta_2$, with $\delta_2\geq\rho^{-2}$ independent of $(t,p,z)\in A$.

Let $c\geq\rho^{-2}$, $(t,p,z)\in A$ and $u \in \aav$. By Assumption \ref{assrho}, let $\tilde u\in \aaa$ be such that
\begin{equation}
\label{Tconverg11}
|J(t,p,z;u)- J(t,p,z;\tilde u)|\leq \eps
\end{equation}
and with the property
\begin{equation}
\label{Tconverg11bis}
Z^{\tilde u}(T)\in [m+\eta(\eps),M-\eta(\eps)]\qquad \ppp\text{-a.s.}
\end{equation}
First of all notice that
\begin{equation}
\label{Tconverg12}
\tilde u \in \aav \subseteq \aaa \setminus \textstyle \bigcup_{i=1}^{m(c)}B_i^{tpz}.
\end{equation}
By (\ref{Tconverg11}) we obtain that
\begin{multline*}
\big|J(t,p,z;u) - J^c(t,p,z; \tilde u)\big| \\
= \big|J(t,p,z;u) - J(t,p,z; \tilde u) - \eee\big[e^{-r(T-t)}\Phi^c(Z^{\tilde u}(T))\big] \big|  \leq \eps + \eee\big[|\Phi^c(Z^{\tilde u}(T))|\big].
\end{multline*}
Notice that by (\ref{Tconverg11bis}) the second term equals zero for $c\geq \eta(\eps)^{-2}$ (in fact $\Phi^c\equiv0$ in $[m+c^{-\frac{1}{2}}, M-c^{-\frac{1}{2}}]$); by recalling (\ref{Tconverg12}), we therefore have that
\begin{equation*}
J(t,p,z;u) \leq J^c(t,p,z; \tilde u) + \eps \leq \sup_{u \in \aaa \setminus \bigcup_{i=1}^{m(c)}B_i^{tpz}}J^c(t,p,z; u)+ \eps,
\end{equation*}
for each $c\geq \max\{\eta(\eps)^{-2},\rho^{-2}\}$. Since this inequality holds for each $(t,p,z)\in A$ and $u \in \aav$, we get (\ref{Tconverg10}).
\end{proof}

\begin{corol}
\label{thconv2}
Let the assumptions of Section \ref{COV} hold. Then the functions $V^c$ converge pointwise to $V$ in $\widetilde\dd$ and $V$ is continuous on $\widetilde \dd$.
\end{corol}

\begin{proof}
It follows immediately from Theorem \ref{thconv1} (recall that $\bigcup_\rho \dd^\rho = \widetilde \dd$).
\end{proof}

\begin{corol}
\label{VVsolvisc}
Let the assumptions of Section \ref{COV} hold. Then the function $V$ is a viscosity solution of Equation (\ref{HJBc}) in $\widetilde \dd$.
\end{corol}

\begin{proof}
Since the functions $V^c$ are viscosity solutions of Equation (\ref{HJBc}) in $[0,T]\times\rr^{n}$ (Theorem \ref{propVc}) and for each $0<\rho<(M-m)/2$ they locally uniformly converge in $\dd^\rho$ to the function $V$ (Theorem \ref{thconv1}), the conclusion follows by the stability property of viscosity solutions with respect to the uniform convergence and the fact that $\widetilde \dd = \bigcup_\rho \dd^\rho$.
\end{proof}

\begin{remark}
We have proved that the value function in the set $\widetilde\dd$ is the locally uniform limit of the functions $V^c$. In some particular cases, stronger conclusions can be achieved: the value function is characterized in its whole domain $\dd$ by the HJB equation. See Section \ref{CAPopzioniswing}.
\end{remark}

\begin{remark}
By standard results in stochastic control theory, it can be proved that the value function does not depend on the probability space we choose.
\end{remark}

We now face the problem of the regularity of the function $V$. In control theory, regularity results are usually achieved by passing to the supremum in estimates on quantities such as $|J(t,p',z;u)-J(t,p'',z;u)|$ or $|J(t,p,z';u)-J(t,p,z'';u)|$, so as to obtain the corresponding inequality for $V$. In the case of constrained problems, this approach cannot be applied to $V(t,p,\cdot)$. In fact, consider $|J(t,p,z';u)-J(t,p,z'';u)|$: on one hand such a quantity is defined only for $u \in \mathcal{A}^{\text{adm}}_{tz'} \cap \mathcal{A}^{\text{adm}}_{tz''}$, on the other hand the supremum should be with respect to different sets (precisely, $\mathcal{A}^{\text{adm}}_{tz'}$ and $\mathcal{A}^{\text{adm}}_{tz''}$). Of course, in particular cases some regularity results can be achieved also for $V(t,p,\cdot)$, see Section \ref{CAPopzioniswing}. The only case when that approach still works regards estimates on $V(t,\cdot,z)$, given that, fixed $t$ and $z$, the set of admissible controls does not depend on $p$. Hence, as for $V(t,\cdot,z)$ we can follow this approach.

\begin{prop}
\label{PrRegVinc1}
Let the assumptions of Section \ref{COV} hold. Assume that there exists a constant $\bar C>0$ such that
\begin{equation}
\label{regvinc1}
|L(t,p,z,v)-L(t,q,z,v)| \leq \bar C|p-q|, \qquad  |\Phi(p,z)-\Phi(q,z)| \leq \bar C|p-q|,
\end{equation}
for each $p,q \in \rr^n$, $t \in[0,T]$, $v \in U$ and $z \in \rr$. Then the function $V(t,\cdot,z)$ is Lipschitz continuous, uniformly in $(t,z)$. Moreover, the gradient $D_p V(t,p,z)$ exists for a.e.~$(t,p,z) \in \dd$ and we have $| D_p V(t,p,z)|\leq M_1$ for some constant $M_1>0$ depending only on $U$, $T$, $\bar C$ and on the constants in (\ref{fs}) and (\ref{cpolfs}).
\end{prop}

\begin{proof}
Let $(t,p,z)\in \dd$, $h>0$, $\xi \in \rr^n$ with $|\xi|=1$ and $u \in \aaa$. In order to avoid ambiguity, we will omit the subscripts in the notation of the mean value (initial data are different, but the probability is obviously the same). By (\ref{regvinc1}) and estimates (D.8) in \cite[Appendix D]{FS} we have
\begin{align}
\label{regvinc2}
& |J(t,p,z;u)-J(t,p+h\xi,z;u)| \notag\\
& \leq \bar C \mathbb{E}\left[\int_t^T \big|P^{t,p;u}(s)-P^{t,p+h\xi;u}(s)\big|ds + \big|P^{t,p;u}(T)-P^{t,p+h\xi;u}(T)\big|\right] \notag \\
& \leq \bar C (T-t+1) \mathbb{E}[ \|P^{t,p;u}(\cdot)-P^{t,p+h\xi;u}(\cdot)\| ] \notag \\
& \leq C_1 \bar C (T+1) |p-(p+h\xi)| \notag \\
& = C_1 \bar C (T+1) h,
\end{align}
for some constant $C_1>0$. Estimate (\ref{regvinc2}) holds for each $u \in \aaa$; thus, it follows that
\begin{equation}
\label{regvinc3}
|V(t,p,z)-V(t,p+h\xi,z)|\leq M_0h,
\end{equation}
where $M_0:=C_1 \bar C (T+1)$. The function $V(t,\cdot,z)$ is therefore Lipschitz continuous, uniformly in $(t,z)$, and then a.e.~differentiable by the Rademacher theorem. By classical results it follows that $D_pV(t,p,z)$ exists for a.e.~$(t,p,z) \in \dd$. Finally, if the gradient exists and $e_i \in \rr^n$ is a vector of the canonical basis ($i=1,\dots,n$), by (\ref{regvinc3}) we get
\begin{equation*}
|(D_pV(t,p,z))_i|= \lim_{h \to 0^+} \frac{|V(t,p,z)-V(t,p+h e_i,z)|}{h} \leq M_0,
\end{equation*}
and then the estimate on the gradient immediately follows.
\end{proof}

\begin{prop}
\label{PrRegVinc2}
Let the assumptions of Section \ref{COV} hold. Assume that $\Phi\in C^2(\rr^{n+1})$, that the functions $f(t,\cdot,v)$, $\sigma(t,\cdot,v)$, $L(t,\cdot,\cdot,v)$ are of class $C^{2}$ for each $(t,v) \in [0,T]\times U$ and that there exist constants $\bar C\geq0$, $j \in \mathbb{N}$ such that
\begin{gather*}
|D_p f(t,p,v)|+|D_p^2f(t,p,v)|+|D_p\sigma(t,p,v)|+|D_p^2\sigma(t,p,v)|\leq \bar C, \\
|D_{(p,z)} L(t,p,z,v)|+|D_{(p,z)}^2L(t,p,z,v)|\leq \bar C (1+|p|^j+|z|^j), \\
|D_{(p,z)}\Phi(p,z)|+|D_{(p,z)}^2\Phi(p,z)|\leq \bar C (1+|p|^j+|z|^j),
\end{gather*}
for each $p\in \rr^n$, $t\in [0,T]$, $v \in U$ and $z \in \rr$. The function $V(t,\cdot,z)$ is then locally semiconvex, uniformly in $t$, and a.e.~twice differentiable.
\end{prop}

\begin{proof}
Since $\Phi\in C^2_p(\rr^{n+1})$, it is possible to rewrite the problem so that $\Phi \equiv 0$ (see \cite[Remark IV.6.1]{FS}). By arguing as in the proof of \cite[Lemma IV.9.1]{FS} (with minor modifications: the assumptions are slightly different), we get
\begin{equation*}
V(t,p+h\xi,z)+V(t,p-h\xi,z)-2V(t,p,z)\geq -M_2(1+|p|^j)h^2,
\end{equation*}
for each $(t,p,z) \in \dd$, $h>0$ and $\xi \in \rr^n$ with $|\xi|=1$, where $M_2>0$ is a constant. The function $V(t,\cdot,z)$ is therefore locally semiconvex, uniformly in $(t,z)$, and then a.e.~twice differentiable by the Alexandrov theorem.
\end{proof}

\subsection{Examples}
\label{COVesempi}

We now show two wide classes of problems satisfying Assumption \ref{assrho}. We first consider problems where $U$ is a compact interval of $\rr$ and $g(s,v)=v$, so that the constraint is $z+\int_t^T u(s) \in [m,M]$.

\begin{prop}
\label{casogen1}
Let $a,b \in \rr$, with $a<b$. Let the assumptions of Section \ref{COV} hold, with $U=[a,b]$ and $g(s,v)=v$. Moreover, assume that there exist $\Gamma, l >0$ such that for $\xi=f,\sigma$ the following condition holds:
\begin{equation}
\label{richiestafsigma}
|\xi(s,p,v')-\xi(s,p,v'')| \leq \Gamma (1+|p|)|v'-v''|^l, \quad \forall (s,p) \in [0,T]\times\rr^n, \,\,\, \forall v',v''\in U.
\end{equation}
Then Assumption \ref{assrho} is satisfied.
\end{prop}

\begin{proof}
For the sake of simplicity, in this proof we assume $l=1$ (for the general case, in the definition of $\tilde u$ it suffices to substitute $\delta^2$ by $\delta^i$, where $i >1/l$).

Let $0<\rho<(M-m)/2$, $A$ be a compact subset of $\dd^\rho$, $R>0$ be such that $\overline{\text{B}(0,R)}\supseteq A$, $\eps>0$. Fix $(t,p,z)\in A$ and $u \in \aav$.

Let $\gamma>0$ (it will be afterwards precisely defined). Since the functions $L$ and $\Phi$ are continuous, there exists $\delta = \delta(\eps,\gamma)>0$ such that
\begin{equation}
\label{Tesempio-1}
|L(s,p',z',v')-L(s,p'',z'',v'')|\leq \frac{\eps}{4T} \,\,\quad \text{and} \,\,\quad |\Phi(p',z')-\Phi(p'',z'')| \leq \frac{\eps}{4},
\end{equation}
for each $s\in [0,T]$, for each $p',p'' \in \overline{\text{B}(0,\gamma)}$ with $|p'-p''|\leq \delta$, for each $z',z''\in [m,M]$ with $|z'-z''|\leq T\delta$ and for each $v',v''\in U$ with $|v'-v''|\leq \delta$.

We now define, starting from $u$, a suitable process $\tilde u$. Let $\Pi^M=\{\omega \in \Omega : Z^u(T)\in ]M-\rho/2,M]\}$ and in $\Pi^M$ let $\tilde u$ be defined in the following way:
\begin{equation*}
\tilde u (s) =
\begin{cases}
u(s)-\delta^{2},  &\text{if $s \in E$}, \\
u(s),  &\text{if $s \in [t,T]\setminus E$},
\end{cases}
\end{equation*}
where $E=E(\omega)=\{s\in[t,T]:u(s)-\delta^2 \in ]a,b[\}=\{s\in[t,T]:u(s)>a+\delta^2\}$. Let $\Pi_m=\{\omega \in \Omega : Z^u(T)\in [m,m+\rho/2[\}$ and in $\Pi_m$ let $\tilde u$ be defined in the following way:
\begin{equation*}
\tilde u (s) =
\begin{cases}
u(s)+\delta^2,  &\text{if $s \in F$}, \\
u(s), &\text{if $s \in [t,T]\setminus F$},
\end{cases}
\end{equation*}
where $F=F(\omega)=\{s\in[t,T]:u(s)+\delta^2 \in ]a,b[\}=\{s\in[t,T]:u(s)<b - \delta^2\}$. Finally, in $\Omega \setminus (\Pi^M\cup\Pi_m)$ let
\begin{equation*}
\tilde u \equiv u.
\end{equation*}
We will show that such a process $\tilde u$ satisfies the required properties.

\textit{Step 1}. We  prove that
\begin{equation}
\label{Tesempio0}
Z^{t,z;\tilde u}(T)\in [m+\eta(\eps),M-\eta(\eps)] \qquad \ppp\text{-a.s.},
\end{equation}
for a suitable function $0<\eta\leq(M-m)/2$ depending only on $\rho, T, R, a, b$.

Consider the case $\omega \in \Pi^M$, i.e.
\begin{equation}
\label{Tesempio0bis}
Z^u(T) \in ]M-\rho/2,M].
\end{equation}
Let us first of all notice that
\begin{equation}
\label{Tesempio0ter}
Z^{\tilde u}(T)= z  + \int_t^T \tilde u(s) ds = z + \int_t^T u(s) ds - \delta^2 \mu(E) = Z^u(T) - \delta^2 \mu(E),
\end{equation}
where $\mu$ denotes the Lebesgue measure in $\rr$. We now look for an estimate for $\mu(E)$. By definition of $E$, we have
\begin{equation*}
\int_t^T u(s) ds = \int_E u(s) ds + \int_{[t,T]\setminus E} u(s) ds \leq b \mu(E) + (a + \delta^2)(T-t-\mu(E))
\end{equation*}
and then
\begin{equation}
\label{Tesempio1}
\mu(E) \in \left[ \frac{\int_t^T u(s)ds - (a+\delta^2)(T-t)}{b-a - \delta^2} , T-t \right] \subseteq \left[ \frac{\rho/2-\delta T}{b-a} ,T \right],
\end{equation}
where the inclusion follows by $z+\int_t^T u(s)ds \geq M - \rho/2$ (since $\omega \in \Pi^M$) and $z \leq M - \rho - a(T-t)$ (since $(t,p,z) \in \dd^\rho$). By possibly decreasing $\delta$ (and the choice depends only on $a,b,\rho,T$), we can assume that the lower bound in (\ref{Tesempio1}) is positive. Recall (\ref{Tesempio0ter}): by (\ref{Tesempio0bis}) and (\ref{Tesempio1}) we get
\begin{equation*}
Z^{\tilde u}(T) \in \left]M-\frac{\rho}{2}-\delta^2T, M - \frac{\delta^2\rho/2-\delta^3T}{b-a}\right] \subseteq \left]m+\frac{\rho}{2}, M - \frac{\delta^2\rho/2-\delta^3T}{b-a}\right],
\end{equation*}
where the inclusion follows by $M-\rho/2 > m + \rho/2$ and by assuming $\delta$ sufficiently small. This estimate holds for each $\omega \in \Pi^M$; by arguing in the same way, for each $\omega \in \Pi_m$ we get
\begin{equation*}
Z^{\tilde u}(T) \in \left[ m + \frac{\delta^2\rho/2-\delta^3 \,T}{b-a} , M-\frac{\rho}{2} \right[.
\end{equation*}
Finally, in $\Omega \setminus (\Pi^M\cup\Pi_m)$ we have $Z^{\tilde u}\in[m+\rho/2,M-\rho/2]$. To summarize, condition (\ref{Tesempio0}) is verified with
\begin{equation*}
\eta(\eps)=\min \left\{ \frac{\rho}{2},\frac{\delta(\eps, \gamma)^2\rho/2-\delta(\eps, \gamma)^3 \,T}{b-a} \right\}.
\end{equation*}

\textit{Step 2}. We still have to prove that
\begin{equation}
\label{Tesempio5}
|J(t,p,z;u)-J(t,p,z;\tilde u)| \leq \eps.
\end{equation}
Let $\Pi \subseteq \Omega$ be defined by
\begin{equation*}
\Pi=\{ \| P^u(\cdot) \|\leq \gamma,\| P^{\tilde u}(\cdot) \|\leq \gamma,\| P^u(\cdot) - P^{\tilde u}(\cdot) \|\leq \delta \};
\end{equation*}
first of all, we set for brevity
\begin{multline*}
\Gamma(t,p,z;u,\tilde u) =\int_t^T |L(s,P^u(s),Z^u(s),u(s))-L(s,P^{\tilde u}(s),Z^{\tilde u}(s),\tilde u(s))|ds \\
+ |\Phi(P^u(T),Z^u(T))-\Phi(P^{\tilde u}(T),Z^{\tilde u}(T))|,
\end{multline*}
and notice that
\begin{equation}
\label{Tesempio6}
|J(t,p,z;u)-J(t,p,z;\tilde u)| \leq \eee [\Gamma(t,p,z;u,\tilde u) \mathbbm{1}_\Pi]+ \eee [\Gamma(t,p,z;u,\tilde u) \mathbbm{1}_{\Pi^c}].
\end{equation}
As for the first term in (\ref{Tesempio6}), for $s \in[t,T]$ we have
\begin{equation*}
|Z^u(s)-Z^{\tilde u}(s)|\leq (s-t)\delta^2\leq T\delta \quad\text{and}\quad |u(s)-\tilde u(s)|\leq \delta^2\leq\delta ,
\end{equation*}
so that by (\ref{Tesempio-1}) it follows that
\begin{equation}
\label{Tesempio7}
\eee [\Gamma(t,p,z;u,\tilde u) \mathbbm{1}_\Pi]\leq \left(\int_0^T \frac{\eps}{4T}ds+\frac{\eps}{4}\right) \ppp( \Pi ) = \frac{\eps}{2}\ppp( \Pi )\leq \frac{\eps}{2}.
\end{equation}
We now consider the second term in (\ref{Tesempio6}). By (D.8) in \cite[Appendix D]{FS}, by (\ref{richiestafsigma}) and (\ref{momenti}) we get
\begin{equation}
\label{Tesempio7bis}
\eee [\|P^u(\cdot)-P^{\tilde u}(\cdot) \|] \leq C_1 (1+|p|)\delta^2,
\end{equation}
for a suitable constant $C_1>0$. By the Markov inequality, (\ref{Tesempio7bis}) and (\ref{momenti}) we then get
\begin{multline}
\label{aggiunta}
\ppp(\Pi^c) \leq \eee[\|P^u(\cdot) \|]\gamma^{-1} +\eee[\|P^{\tilde u}(\cdot) \|]\gamma^{-1} + \eee[\|P^u(\cdot)-P^{\tilde u}(\cdot) \|]\delta^{-1} \\
\leq C_2(1+|p|)(\gamma^{-1}+ \delta),
\end{multline}
where $C_2>0$ is a constant. By the H\"{o}lder inequality (twice), estimates as in (\ref{Vpol}) and (\ref{aggiunta}), we obtain that
\begin{align*}
\eee [ \Gamma(t,p,z;u,\tilde u) \mathbbm{1}_{\Pi^c} ]^2 & \leq \eee [ \Gamma(t,p,z;u,\tilde u)^2] \ppp(\Pi^c) \\
& \leq C_3 (1+|p|^{2k}+|z|^{2k})(1+|p|)(\gamma^{-1}+\delta(\eps, \gamma)) \\
& \leq C_4 (1+R^{2k+1})\gamma^{-1}+ C_4 (1+R^{2k+1}) \delta(\eps,\gamma),
\end{align*}
with $C_3,C_4> 0$. First by choosing a suitable $\gamma$ and then by possibly taking a less $\delta$ (and these choices depends only on $R$ and $\eps$), we get
\begin{equation}
\label{Tesempio8}
\eee \left[ \Gamma(t,p,z;u,\tilde u) \mathbbm{1}_{\Pi^c} \right] \leq \frac{\eps}{2}.
\end{equation}
Estimates (\ref{Tesempio6}), (\ref{Tesempio7}) e (\ref{Tesempio8}) imply (\ref{Tesempio5}), thus ending the proof.
\end{proof}

Let us now consider problems where $U$ is a closed ball of $\rr^l$ and $g(s,v)=|v|^p$, so that the constraint is $z+\int_t^T |u(s)|^p \in [m,M]$, with $p \geq 1$.

\begin{prop}
\label{casogen2}
Let $p\geq1$ and $b>0$. Let the assumptions of Section \ref{COV} and (\ref{richiestafsigma}) hold, with $U=\overline{\text{B}(0,b)}\subseteq \rr^l$ and $g(s,v)=|v|^p$. Assumption \ref{assrho} is then satisfied.
\end{prop}

\begin{proof}
The proof is similar to the one of Proposition \ref{casogen1}, with the following modifications:
\begin{itemize}
\item[-] In the definitions of $E$ and $F$ replace $u(s)$ by $|u(s)|$. Process $\tilde u$ in $E$ is now
    \begin{equation*}
    \tilde u(s)= u(s)- \delta^2 \frac{u(s)}{|u(s)|}.
    \end{equation*}
    Notice that $|\tilde u (s)| = |u(s)| - \delta^2$. Similarly in $F$.
\item[-] It is easy to check that
    \begin{align*}
    (\zeta-\delta^2)^p\leq \zeta^p - \delta^{2p}, & \qquad \text{for}\,\, \zeta\geq\delta^2, \\
    (\zeta+\delta^2)^p\geq \zeta^p + \delta^{2p}, & \qquad \text{for}\,\, \zeta\geq0.
    \end{align*}
    By the first estimate, in $\Pi^M$ we have that
    \begin{align*}
    Z^{\tilde u}(T) & = z  + \int_t^T |\tilde u(s)|^p ds  \\
    & = z + \int_E (|u(s)|-  \delta^2)^p ds + \int_{[t,T]\setminus E} |u(s)|^p ds \\
    & \leq z + \int_t^T |u(s)|^p ds - \delta^{2p} \mu(E),
    \end{align*}
    and then we can argue as in the proof of Proposition \ref{casogen1}. As for $\Pi_m$, use the second estimate and the same argument.
\end{itemize}
\end{proof}

\section{Swing contracts with strict constraints}
\label{CAPopzioniswing}

We now use the results of Section \ref{CAPcontrottvinc} to study the problem of optimally exercising swing contracts with strict constraints (see the Introduction). In this case we will obtain %achieve
results stronger than the general ones proved in Section \ref{COVpropfv}.

\subsection{Formulation of the problem}
\label{SWING}

Let $T>0$, ($\Omega$, $\mathcal{F}_T$, $\{\mathcal{F}_s\}_{s\in[0,T]}$, $\mathbb{P}$), $U$, $\aaa$, $P^{t,p}$ and $Z^{t,z;u}$ be as in Section \ref{CAPopzioniswingpenal}. If $(t,p,z)\in[0,T]\times \rr^2$ and $s \in [t,T]$, recall, in particular, that $P^{t,p}(s)$ models the price of energy at time $s$ and that $Z^{t,z;u}(s)$ represents the energy bought up to time $s$, where $u \in \aaa$ is the usage strategy from time $t$ on.

Given $m,M\geq0$ with $m<M$, here we ask the following constraint to hold:
\begin{equation*}
Z^{t,z;u}(T)\in[m,M] \qquad \ppp\text{-a.s.}
\end{equation*}
The problem of the optimal exercise of this contract (i.e.~to find a process $u$ satisfying all the conditions and providing the maximal expected earning) is clearly a constrained stochastic optimal control problem as described in Section \ref{COV} (here $P$ does not depend on $u$), whose value function is
\begin{equation*}
V(t,p,z)=\sup_{u \in \aav}\eee\left[\int_t^T e^{-r(s-t)}(P^{t,p}(s)-K)u(s) ds\right].
\end{equation*}

In this problem the sets $\dd, \widetilde\dd, \dd^\rho$ are formed by all points $(t,p,z)\in[0,T]\times\rr^2$ such that $(t,z)$ belongs, respectively, to the marked surfaces in Figure \ref{PICsets}.
\begin{figure}[htb]
\centering
\includegraphics[width=0.7\columnwidth]{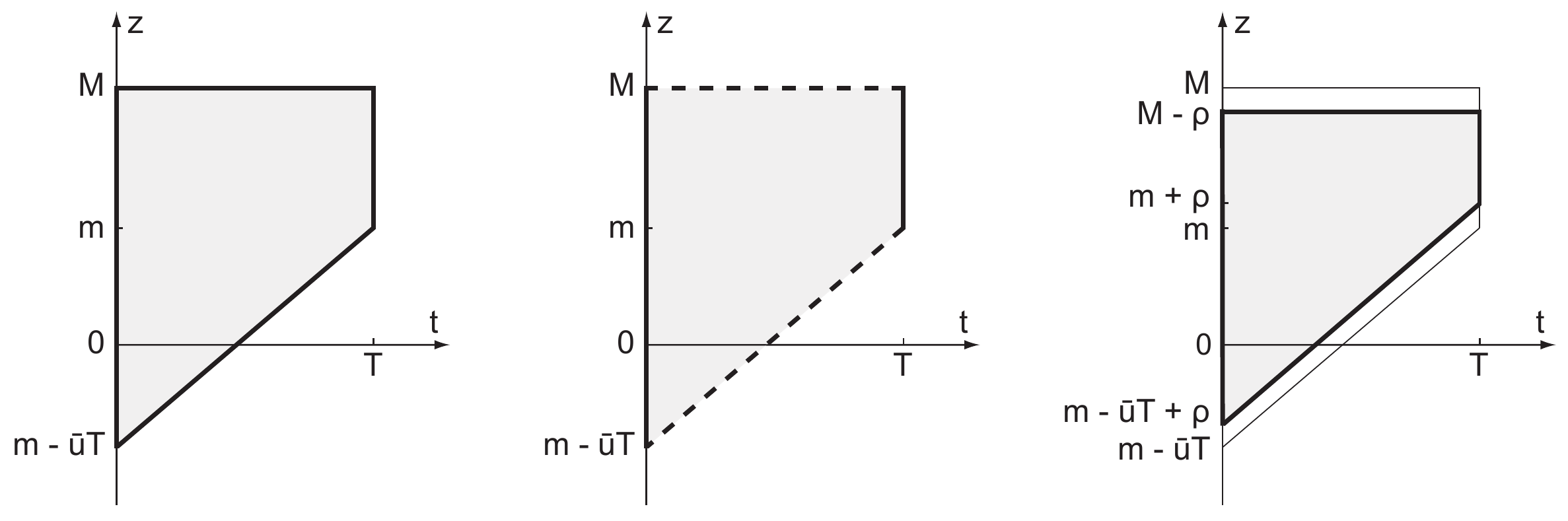}
\caption{the sets $\dd, \widetilde\dd, \dd^\rho$}
\label{PICsets}
\end{figure}
\\ More in details, we have
\begin{gather*}
\dd = \{(t,p,z) \in [0,T]\times \rr^2 : m-\bar u (T-t)\leq z \leq M\}, \\
\widetilde \dd = \{(t,p,z) \in [0,T]\times \rr^2 : m-\bar u (T-t)< z < M\},\\
\dd^\rho = \{(t,p,z) \in [0,T]\times \rr^2 : m+\rho-\bar u (T-t)\leq z \leq M-\rho\}.
\end{gather*}
Notice that these sets include initial data that are inconsistent with the practical problem: in fact, our mathematical formulation %model
admits negative starting values for $p$ and $z$.

The functions $V^c$ are here defined by
\begin{equation*}
V^c(t,p,z)=\sup_{u \in \aaa}\eee\left[\int_t^T e^{-r(s-t)}(P^{t,p}(s)-K)u(s) ds + e^{-r(T-t)}\Phi^c(Z^{t,z;u}(T))\right],
\end{equation*}
for each $c>0$ and $(t,p,z)\in[0,T]\times \rr^{2}$, where $\Phi^c$ is defined in (\ref{phic}). This has also a nice economical interpretation: in fact, here we are approximating a swing contract with the strict constraint $Z(T) \in [m,M]$ with a sequence of suitable contracts with increasing penalties for $Z(T) \notin \left[m + \frac{1}{\sqrt{c}}, M - \frac{1}{\sqrt{c}}\right]$.

The HJB equation for the function $V^c$ is
\begin{multline}
\label{equaz}
-V^c_t(t,p,z) + r V^c(t,p,z) - f(t,p)V^c_p(t,p,z) - \frac{1}{2}\sigma^2(t,p)V_{pp}^c(t,p,z) \\
+ \min_{v \in [0,\bar u]}\left[ -v (V^c_z(t,p,z) + p-K)  \right]=0, \qquad \forall (t,p,z) \in [0,T[\times\rr^{2},
\end{multline}
with final condition
\begin{equation*}
V^c(T,p,z)=\Phi^c(z), \qquad \forall (p,z) \in \rr^{2}.
\end{equation*}

\subsection{Properties of the value function}
\label{SWINGprop}

The problem described in Section \ref{SWING} belongs to the class treated in Proposition \ref{casogen1}. Therefore Theorem \ref{thconv1}, Corollary \ref{thconv2} and Corollary \ref{VVsolvisc} hold, but it turns out that in this case we can strengthen such results.

We set for brevity
\begin{equation*}
\alpha=\{(t,p,z) \in \dd : z=M \}, \,\,\, \beta=\{(t,p,z)\in \dd : z+\bar u (T-t)=m \}, \,\,\, \gamma = \{T\} \times \rr \times [m,M],
\end{equation*}
so that $\dd \setminus \widetilde \dd = \alpha \cup \beta$.

Let us first consider Theorem \ref{thconv1} and adapt it to our problem, as here something about $\dd \setminus \widetilde \dd$ can also be said.

\begin{prop}
\label{PrStrict1}
Let the assumptions of Section \ref{SWING} hold. The functions $V^c$ converge to $V$ uniformly on compact subsets of $\widetilde \dd$ . Moreover, if $(t,p,z)\in \alpha$ we have $V(t,p,z)=0$. Finally, if $(t,p,z)\in \beta$ we have
\begin{equation}
\label{Tcont0}
V(t,p,z)=\bar u \, \eee\left[\int_t^T e^{-r(s-t)}(P^{t,p}(s)-K) ds  \right] =:\xi(t,p).
\end{equation}
\end{prop}

\begin{proof}
As for the first part, notice that each compact subset of $\widetilde \dd$ is contained in some $\dd^\rho$ and use Theorem \ref{thconv1}. Second and third items: in $\alpha \cup \beta$ there exists a unique admissible control, respectively $u\equiv 0$ and $u \equiv \bar u$.
\end{proof}

Notice that the boundary condition %function
$\xi$ in (\ref{Tcont0}) is continuous and can be computed in many models used in practice (see \cite{BLN}).

Corollary \ref{thconv2} assures continuity of $V$ on $\widetilde \dd$. We now prove that in this case a stronger result holds, i.e. the value function is continuous on the whole domain $\dd$. For this, we first need a technical lemma, where we give a bound for the mean distance between solutions of (\ref{Pcasopartic}) starting from different data.

\begin{lemma}
\label{lemma}
Let the assumptions of Section \ref{SWING} hold. Let $t_1,t_2\in[0,T]$ with $t_1<t_2$ and $p_1,p_2\in \rr^n$. Then
\begin{equation*}
\mathbb{E}\big[\big|P^{t_1,p_1}(s)-P^{t_2,p_2}(s)\big|\big] \leq M\big[|p_2-p_1|+(t_2 -t_1)^{\frac{1}{2}}(1+|p_1|)\big],
\end{equation*}
for each $s \in [t_2,T]$, where $\mathbb{E}$ denotes the mean value with respect to the probability $\mathbb{P}$ and $M>0$ is a constant depending only on $T,U$ and on the constants in (\ref{fs}) and (\ref{cpolfs}).
\end{lemma}

\begin{proof}
Let $t_1,t_2\in[0,T]$ with $t_1<t_2$, $s\in[t_2,T]$ and $p_1,p_2\in \rr^n$. By standard estimates on stochastic differential equations (see e.g.~\cite[Appendix D]{FS} or \cite{KR}) we arrive at
\begin{multline*}
\mathbb{E}\big[| P^{t_1,p_1} (s) - P^{t_2,p_2}(s)|^2\big] \leq \big(C_1|p_2-p_1|^2 + C_1(t_2-t_1)(1+|p_1|^2)\big) e^{C_1(s-t_2)} \\
\leq C_2|p_2-p_1|^2+C_2(t_2-t_1)(1+|p_1|^2),
\end{multline*}
where $C_1,C_2>0$ are constants. Hence, by the H\"{o}lder inequality
\begin{equation*}
\mathbb{E}\big[| P^{t_1,p_1} (s) - P^{t_2,p_2} (s)|\big] \leq\mathbb{E}\big[| P^{t_1,p_1} (s) - P^{t_2,p_2}(s)|^2\big] ^{\frac{1}{2}} \leq C_3|p_2-p_1|+ C_3(t_2-t_1)^{\frac{1}{2}}(1+|p_1|),
\end{equation*}
for some constant $C_3>0$.
\end{proof}

\begin{prop}
\label{PrStrict2}
Let the assumptions of Section \ref{SWING} hold. Then $V$ is continuous on $\dd$.
\end{prop}

\begin{proof}
As Corollary \ref{thconv2} holds, we have to prove that $V$ is continuous on $\dd \setminus \widetilde \dd = \alpha \cup \beta$.

\textit{Step 1: continuity on $\alpha$.} Let $(\tilde t, \tilde p, \tilde z) \in \alpha$. Since in this case the only admissible control is $u\equiv 0$, we have to prove that
\begin{equation}
\label{Tcont1}
\lim_{\substack{(t,p,z)\to(\tilde t, \tilde p, \tilde z) \\ (t,p,z)\in \dd}} V(t,p,z) = V(\tilde t, \tilde p, \tilde z) = 0.
\end{equation}

Let $(t,p,z)\in \dd$ and $u \in \aav$. Given arbitrary $\gamma>0$, we first of all observe that
\begin{multline}
\label{Tcont2}
\eee\left[\left(\int_t^T |P^{t,p}(s)-K|u(s) ds \right)\mathbbm{1}_{ \{ \|P(\cdot)\|\leq\gamma\} } \right]  \\
\leq (\gamma+K)\eee\left[\int_t^T u(s) ds \right] \leq (\gamma+K) (M-z),
\end{multline}
where in the last passage we have used condition $Z^u(T)\leq M$. By the H\"{o}lder inequality (twice), estimates as in (\ref{Vpol}), the Markov inequality and (\ref{momenti}) we get
\begin{multline}
\label{Tcont3}
\eee\left[\left(\int_{t}^T |P^{t,p}(s)-K|u(s) ds \right)\mathbbm{1}_{ \{ \|P^{t,p}(\cdot)\|>\gamma\} } \right] \\
\leq T\eee \left[\int_{t}^T (P^{t,p}(s)-K)^2 u(s)^2 ds \right]^{\frac{1}{2}} \ppp(\|P^{t,p}(\cdot)\|>\gamma)^{\frac{1}{2}} \leq C_1 (1+|p|)^{\frac{3}{2}}\gamma^{-\frac{1}{2}},
\end{multline}
for some constant $C_1>0$. By (\ref{Tcont2}) and (\ref{Tcont3}) it follows that
\begin{equation}
\label{Tcont4}
\biggl|\sup_{u \in \aav} \eee\left[\int_t^T e^{-r(s-t)}(P^{t,p}(s)-K)u(s) ds\right] \biggr|\\
\leq (\gamma+K)( M-z)+ C_1 (1+|p|)^{\frac{3}{2}}\gamma^{-\frac{1}{2}}.
\end{equation}
Inequality (\ref{Tcont4}) holds for each $\gamma>0$ and for each $(t,p,z)\in \dd$. We get (\ref{Tcont1}) by passing to the limit first as $(t,p,z)\to(\tilde t, \tilde p, \tilde z)$ (recall that $\tilde z =M$) and then as $\gamma \to \infty$.

\textit{Step 2: continuity on $\beta$.} Let $(\tilde t, \tilde p, \tilde z) \in \beta$. Since in $\beta$ function $V$ is as in (\ref{Tcont0}), we have to prove that
\begin{equation}
\label{Tcont5}
\lim_{\substack{(t,p,z)\to(\tilde t, \tilde p, \tilde z) \\ (t,p,z)\in \dd}} V(t,p,z) = V(\tilde t, \tilde p, \tilde z) = \bar u \, \mathbb{E}_{\tilde t \tilde p \tilde z}\left[\int_{\tilde t}^T e^{-r(s-\tilde t)}(P^{\tilde t, \tilde p}(s)-K) ds  \right].
\end{equation}

From now on, we will omit the subscripts in the notation of the mean value (the initial data are different, but the probability is clearly the same). Let $(t,p,z)\in \dd$ (notice that necessarily $t \leq \tilde t$) and fix $u \in \aav$; for simplicity we will write $P = P^{t,p}$ and $\tilde P = P^{\tilde t, \tilde p}$. Since for $s\in [\tilde t,T]$ we have
\begin{multline*}
e^{-r(s-t)} (P(s)-K)u(s)-e^{-r(s-\tilde t)}(\tilde P(s)-K)\bar u \\
= e^{-r(s-t)}(P(s)-\tilde P(s))u(s)-e^{-r(s-t)}(\tilde P(s)-K)(\bar u -u(s)) \\
- (e^{-r(s-\tilde t)}-e^{-r(s- t)})(\tilde P(s)-K)\bar u ,
\end{multline*}
let us first of all observe that
\begin{align}
\label{Tcont7}
& \mathbb{E}\biggr[\biggr|\int_t^T e^{-r(s- t)}(P(s)-K)u(s) ds  - \bar u \int_{\tilde t}^T e^{-r(s- \tilde t)}(\tilde P(s)-K) ds  \biggr| \biggr] \nonumber \\
& \leq \mathbb{E}\biggr[\int_t^{\tilde t} |P(s)-K|u(s) ds \biggr]  + \mathbb{E}\biggr[\int_{\tilde t}^T \big| e^{-r(s-t)} (P(s)-K)u(s)-e^{-r(s-\tilde t)}(\tilde P(s)-K)\bar u \big| ds  \biggr] \nonumber \\
& \leq \mathbb{E}\biggr[\int_t^{\tilde t} |P(s)-K|u(s) ds\biggr] + \mathbb{E}\biggr[\int_{\tilde t}^T |P(s)-\tilde P(s)|u(s)ds\biggr]  \\
& \qquad+ \mathbb{E}\biggr[\int_{\tilde t}^T  |\tilde P(s)-K|(\bar u -u(s)) ds \biggr] + \mathbb{E}\biggr[\int_{\tilde t}^T  \big(e^{-r(s-\tilde t)}-e^{-r(s- t)}\big)|\tilde P(s)-K|\bar u ds \biggr]. \nonumber
\end{align}
Consider the first term in (\ref{Tcont7}). By estimates as in (\ref{Vpol}) we have that
\begin{equation}
\label{Tcont8}
\mathbb{E}\biggr[\int_t^{\tilde t} |P(s)-K|u(s) ds\biggr] \leq C_2(\tilde t - t)(1+|p|),
\end{equation}
for some constant $C_2>0$. As for the second term in (\ref{Tcont7}), by the Fubini-Tonelli theorem and Lemma \ref{lemma} we get
\begin{equation}
\mathbb{E}\biggr[\int_{\tilde t}^T |P(s)-\tilde P(s)|u(s)ds\biggr] \leq C_3\big[|p-\tilde p|+(\tilde t -t)^{\frac{1}{2}}(1+|p|)\big],
\end{equation}
where $C_3>0$ is a constant. Let us now estimate the third term in (\ref{Tcont7}). Given arbitrary $\gamma>0$, we observe that
\begin{multline}
\mathbb{E}\left[\left(\int_{\tilde t}^T |\tilde P(s)-K|(\bar u - u(s)) ds \right)\mathbbm{1}_{ \{ \|\tilde P(\cdot)\|\leq\gamma\} } \right]
\leq (\gamma+K) \mathbb{E}\left[\int_{\tilde t}^T ( \bar u - u(s)) ds \right] \\
= (\gamma+K) \mathbb{E}\left[\bar u (T-t) - \int_t^T u(s) ds \right] \leq (\gamma+K)\big(\bar u (T-t) - m + z\big),
\end{multline}
where in the last passage we have used condition $Z^u(T)\geq m $. By arguing as in (\ref{Tcont3}), we get
\begin{equation}
\mathbb{E}\left[\left(\int_{\tilde t}^T |\tilde P(s)-K|(\bar u - u(s)) ds \right)\mathbbm{1}_{ \{ \|\tilde P(\cdot)\|>\gamma\} } \right] \leq C_4 (1+|\tilde p|)^{\frac{3}{2}}\gamma^{-\frac{1}{2}},
\end{equation}
for some constant $C_4>0$. We finally consider the fourth term in (\ref{Tcont7}). By local Lipschitzianity of the exponential function and by estimates as in (\ref{Vpol}) we obtain
\begin{equation}
\label{Tcont12}
\mathbb{E}\biggr[\int_{\tilde t}^T  \big(e^{-r(s-\tilde t)}-e^{-r(s- t)}\big)|\tilde P(s)-K|\bar u ds \biggr] \leq C_5(\tilde t - t)(1+|\tilde p|),
\end{equation}
where $C_5>0$ is constant.

By estimates from (\ref{Tcont8}) to (\ref{Tcont12}), it follows from (\ref{Tcont7}) that
\begin{multline}
\label{Tcont13}
\biggr|\sup_{u \in \aav}\mathbb{E}\biggr[\int_t^T e^{-r(s-t)}(P(s)-K)u(s) ds  \biggr]- \bar u \, \mathbb{E}\biggr[\int_{\tilde t}^T e^{-r(s-\tilde t)}(\tilde P(s)-K) ds  \biggr] \biggr|  \\
\leq C_2(\tilde t - t)(1+|p|) + C_3\big [|p-\tilde p|+(\tilde t -t)^{\frac{1}{2}}(1+|p|)\big] + C_5(\tilde t - t)(1+|\tilde p|) \\
+ (\gamma+K)\big(\bar u (T-t) - m + z \big) + C_4 (1+|\tilde p|)^{\frac{3}{2}}\gamma^{-\frac{1}{2}}.
\end{multline}
Estimate (\ref{Tcont13}) holds for each $\gamma>0$ and for each $(t,p,z)\in \dd$. We get (\ref{Tcont5}) by passing to the limit first as $(t,p,z)\to(\tilde t, \tilde p, \tilde z)$ (recall that $\tilde z+\bar u (T-\tilde t)=m$) and then as $\gamma \to \infty$.
\end{proof}

Let us now consider the HJB equation and prove a result which is stronger than Corollary \ref{VVsolvisc}: in this case the value function is, in its whole domain $\dd$, the unique viscosity solution of the HJB equation with polynomial growth and the boundary conditions given below. Thus, we get another characterization of the value function, in addition to %besides
the one of Proposition \ref{PrStrict1}.

\begin{thm}
\label{PrStrict3}
Let the assumptions of Section \ref{SWING} hold. Then the function $V$ is the unique viscosity solution of Equation (\ref{equaz}) in the domain $\dd  \setminus (\alpha \cup \beta \cup \gamma)$, with boundary conditions
\begin{align}
\label{HJB2swing}
& V(t,p,z)=0, \quad\qquad \forall(t,p,z) \in \alpha, \nonumber \\
& V(t,p,z)=\xi(t,z), \quad\qquad \forall(t,p,z) \in \beta, \\
& V(T,p,z)=0, \quad\qquad \forall(p,z) \in \rr \times [m,M], \nonumber
\end{align}
such that
\begin{equation}
\label{cpolswing}
|V(t,p,z)| \leq \check C (1+|p|^2+|z|^2), \qquad \forall (t,p,z) \in \dd,
\end{equation}
for some constant $\check C >0$.
\end{thm}

\begin{proof}
In this problem, $k=2$ in (\ref{VVpol}). Thus, by (\ref{VVpol}), Corollary \ref{VVsolvisc} and Proposition \ref{PrStrict1}, the function $V$ is viscosity a solution of problem (\ref{equaz})-(\ref{HJB2swing})-(\ref{cpolswing}). We now need a uniqueness result. By the following change of variables
\begin{equation*}
t'=t, \qquad p'=p, \qquad z'= \frac{z-M}{M-m+\bar u (T-t)}+1,
\end{equation*}
problem (\ref{equaz})-(\ref{HJB2swing}) becomes
\begin{multline*}
-V_{t'}(t',p',z') -\frac{\bar u (z'-1)}{M-m+\bar u (T-t')} V_{z'}(t',p',z') \\
+ r V(t',p',z') - f(t',p')V_{p'}(t',p',z') - \frac{1}{2}\sigma^2(t',p')V_{p'p'}(t',p',z') \\
+ \min_{v \in [0,\bar u]}\left[ -v \left(\frac{1}{M-m+\bar u (T-t')}V_{z'}(t',p',z') + p'-K\right)  \right]=0, \\
\forall (t',p',z') \in [0,T[ \times \rr \times ]0,1[,
\end{multline*}
with boundary condition
\begin{align*}
& V(T,p',z')=0, \quad\qquad \forall(p',z')\in \rr \times [0,1], \\
& V(t',p',1)=0, \quad\qquad \forall(t',p')\in [0,T] \times \rr, \\
& V(t',p',0)=\xi(t',p'), \quad\qquad \forall(t',p')\in [0,T] \times \rr.
\end{align*}
Notice that the polynomial growth is preserved and that now the domain is $[0,T[ \times \rr \times ]0,1[$. By arguing as in \cite{DaLioLey2}, one can prove that uniqueness holds for such a problem.
\end{proof}

The previous result generalizes an analogous result in \cite{BLN}, valid in the case $m = 0$. We now turn to prove some properties of the value function with respect to the variables $p$ and $z$. As for $V(t,\cdot,z)$, Propositions \ref{PrRegVinc1} and \ref{PrRegVinc2} hold.

\begin{prop}
\label{PrStrict4}
Let the assumptions of Section \ref{SWING} hold. Let $(t,z)\in[0,T]\times \rr$ be such that $(t,p,z) \in \dd$ for each $p \in \rr$. Then
\begin{itemize}
\item[-] the function $V(t,\cdot,z)$ is Lipschitz continuous, uniformly in $(t,z)$. Moreover, the derivative $V_p(t,p,z)$ exists for a.e.~$(t,p,z) \in \dd$ and we have $| V_p(t,p,z)|\leq M_1$, for some constant $M_1>0$ depending only on $T$, $\bar u$ and on the constants in (\ref{fs}), (\ref{cpolfs}) and (\ref{regvinc1}).
\item[-] if $f(s,\cdot),\sigma(s,\cdot) \in C^{2}_b(\rr)$, uniformly in $s \in [0,T]$, the function $V(t,\cdot,z)$ is locally semiconvex, uniformly in $t$, and a.e.~twice differentiable.
\end{itemize}
\end{prop}

\begin{proof}
The first part follows from Proposition \ref{PrRegVinc1} (notice that function $p \mapsto (p-K)v$ is Lipschitz continuous). As for the second item, it suffices to rewrite Proposition \ref{PrRegVinc2} (notice that the function $p \mapsto (p-K)v$ is of class $C^\infty(\rr)$, with bounded derivatives).
\end{proof}

Let us now consider the function $V(t,p, \cdot)$. Recall that its domain is $[m-\bar u (T-t),M]$.

\begin{prop}
\label{PrStrict5}
Let the assumptions of Section \ref{SWING} hold. For each $(t,p) \in [0,T]\times \rr$ the function $V(t,p,\cdot)$ is
\begin{itemize}
\item[-] concave, Lipschitz continuous and a.e.~twice differentiable;
\item[-] weakly increasing in $[m-(T-t)\bar u, M-(T-t)\bar u]$ and weakly decreasing in $[m,M]$. In particular, if $M-(T-t)\bar u \geq m$ then the function $V(t,p,\cdot)$ is constant in $[m, M-(T-t)\bar u]$ (they all are maximum points).
\end{itemize}
\end{prop}

\begin{proof}
\textit{Item 1} (this is an adaptation of \cite[Prop.~3.4]{BLN}, which takes into account only an upper bound on $Z^{t,z;u}(T)$). Let $(t,p) \in [0,T]\times\rr$, $z_1,z_2\in [m-\bar u (T-t),M]$, $u_1 \in \mathcal{A}^{\text{adm}}_{t z_1}$ e $u_2 \in \mathcal{A}^{\text{adm}}_{t z_2}$. By (\ref{stima2}) the process $(u_1+u_2)/2$ belongs to the set of admissible controls for initial point $(t,(z_1+z_2)/2)$. By the linearity of the function $v \mapsto (P^{t,p}(s)-K)v$ we have
\begin{equation}
\label{regstr1}
\frac{J(t,p,z_1;u_1)+J(t,p,z_2;u_2)}{2} =  J \left(t,p,\frac{z_1+z_2}{2};\frac{u_1+u_2}{2}\right) \leq V\left(t,p,\frac{z_1+z_2}{2}\right).
\end{equation}
Since (\ref{regstr1}) holds for each $u_1 \in \mathcal{A}^{\text{adm}}_{t z_1}$ and $u_2 \in \mathcal{A}^{\text{adm}}_{t z_2}$, it follows that
\begin{equation*}
\frac{V(t,p,z_1)+V(t,p,z_2)}{2} \leq V\left(t,p,\frac{z_1+z_2}{2}\right),
\end{equation*}
which implies the concavity of the function $V(t,p,\cdot)$. Local Lipschitzianity is a well-known property of concave functions (and here the domain is a compact set), while the a.e.~existence of the second derivative follows from the Alexandrov theorem.

\textit{Item 2.} Let $(t,p) \in [0,T] \times \rr$. If $m \leq z_1 \leq z_2 \leq M$, it is easy to check that $\mathcal{A}^{\text{adm}}_{t z_2} \subseteq \mathcal{A}^{\text{adm}}_{t z_1}$, so that $ V(t,p,z_1)\geq V(t,p,z_2)$. Similarly, if $m-(T-t)\bar u \leq z_1 \leq z_2 \leq M-(T-t)\bar u$, we have $\mathcal{A}^{\text{adm}}_{t z_1} \subseteq \mathcal{A}^{\text{adm}}_{t z_2}$ and then $ V(t,p,z_2)\leq V(t,p,z_1)$. The second part immediately follows, since $[m-(T-t)\bar u, M-(T-t)\bar u] \cap [m,M]=[m, M-(T-t)\bar u]$.
\end{proof}

The monotonicity result in Proposition \ref{PrStrict5} is described in Figure \ref{PICcrescconstr}.
\begin{figure}[htb]
\centering
\includegraphics[width=0.25\columnwidth]{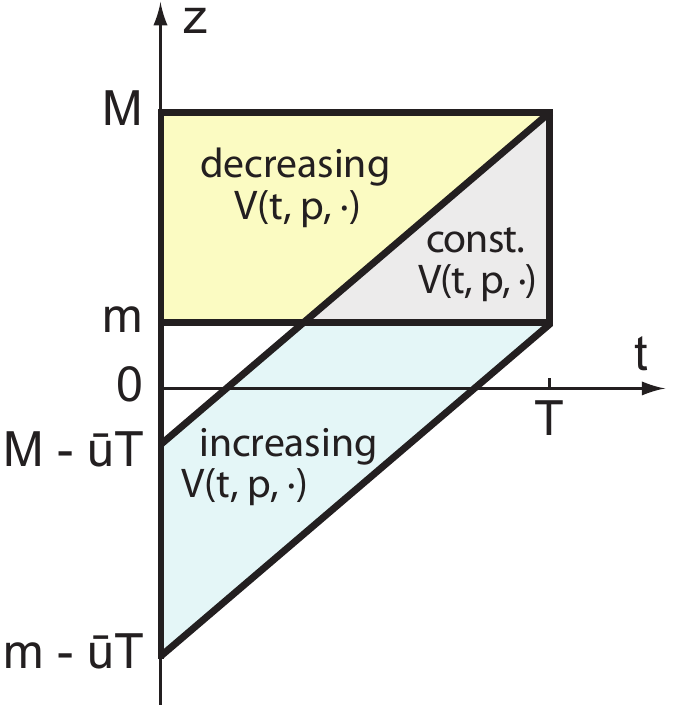}
\caption{monotonicity of $V(t,p,\cdot)$}
\label{PICcrescconstr}
\end{figure}

As in Section \ref{CAPopzioniswingpenal}, it was foreseeable that the function $V(t,p,\cdot)$ is constant in an interval: if $M-(T-t)\bar u \geq m$ and $z \in [m, M-(T-t)\bar u]$ then $\aav=\aaa$ (i.e.~all controls satisfies the constraint), which implies that the initial value $z$ does not influence the value function. This generalizes an intuitive result in \cite[Lemma 3.2]{BLN}: for $(t,z)$ such that the volume constraint is {\em de facto} absent, the value function $V$ does not depend on $z$.

Finally, also in this case Remark \ref{remark} holds: by Proposition \ref{PrStrict5} the candidate in (\ref{candidate}-\ref{candidate2}) is well-defined.

\section{Conclusions}

We characterize the value of swing contracts in continuous time as the unique viscosity solution of a Hamilton-Jacobi-Bellman equation with suitable boundary conditions. More in details, swings can be divided in two broad contract classes, those with penalties on the cumulated quantity of energy $Z(T)$ at the end $T$ of the contract, and those with strict constraints on the same quantity: usually these constraints and penalties are meant to make $Z(T)$ belong to an interval $[m,M]$ with $m > 0$ (in real contracts usually $m > 0.8M$, see \cite{Loland}). 

In Section 2 we treat the case of contracts with penalties, which results in a straightforward application of classical optimal control theory, and in that case only a terminal condition is needed. For swing contracts with penalties, we prove that their value is the unique viscosity solution of the HJB equation (\ref{HJBpenalt}), and that is Lipschitz both in $p$ (spot price of energy) as in $z$ (current cumulated quantity), with first weak derivatives with sublinear growth. We also prove that the value function is also concave with respect to $z$, weakly increasing for $z \leq M - (T-t)\bar u$, where $t$ is the current time and $\bar u$ is the maximum marginal energy that can be purchased, and weakly decreasing for $z \geq m$. In this, we extend and generalize previous results of \cite{BLN}, which were proved only for swing contracts with strict penalties. These results make the candidate optimal exercise policy in Equations (\ref{candidate}--\ref{candidate2}) well defined. 

Conversely, the case of contracts with strict constraints gives rise to a stochastic control problem with a nonstandard state constraint in $Z(T)$. In Section 3 we approach a suitable generalization of this problem by a penalty method: we consider a general constrained problem and approximate the value function with a sequence of value functions of appropriate unconstrained problems with a penalization term in the objective functional, showing that they converge uniformly on compact sets to the value function of the constrained problem. 

In Section 4 we come back to the case of swing contracts with strict constraints: in this case the penalty functions used in Section 3 turn out to be penalties of suitable swing contracts, so that we also have the economic interpretation that a swing contract with strict constraints can be approximated by swing contracts with suitable penalties. In this context we succeed in strengthening the results of Section 3, by characterizing the value function as the unique viscosity solution with polynomial growth of the HJB equation (\ref{HJBpenalt}) subject to the boundary conditions in Equation (\ref{HJB2swing}). As for the smoothness of the value function with respects to $p$ and $z$, we find exactly the same results as in Section 2, extending previous results of \cite{BLN} to the case $m > 0$. These results make the candidate optimal exercise policy in Equations (\ref{candidate}--\ref{candidate2}), i.e. the same as in Section 2, again well defined.


\begin{thebibliography}{11}

\bibitem{BBP} O.~Bardou, S.~Bouthemy, G.~Pagès, \textit{Optimal quantization for the pricing of swing options}, Appl. Math. Finance 16 (2009), no. 1-2, 183--217.

\bibitem{Barrera} C.~Barrera-Esteve, F.~Bergeret, C.~Dossal, E.~Gobet, A.~Meziou, R.~Munos, D.~Reboul-Salze, \textit{Numerical methods for the pricing of swing options: a stochastic control approach}, Methodol. Comput. Appl. Probab. 8 (2006), no. 4, 517--540.

\bibitem{BLN} F.~E.~Benth, J.~Lempa, T.~K.~Nilssen, \textit{On the optimal exercise of swing options in electricity markets}, The Journal of Energy Markets 4 (2012), no. 4,  1--27.

\bibitem{ct} R.~Carmona, N.~ Touzi, \textit{Optimal multiple stopping and valuation of swing options}, Math. Finance 18 (2008), no. 2, 239--268.

\bibitem{USER} M.~G.~Crandall, H.~Ishii, P.-L.~Lions, \textit{User's guide to viscosity solutions of second order partial differential equations}, Bull. Amer. Math. Soc. (N.S.) 27 (1992), no. 1, 1--67.

\bibitem{DaLioLey2} F.~Da Lio, O.~Ley, \textit{Convex Hamilton-Jacobi equations under superlinear growth conditions on data}, Appl. Math. Optim. 63 (2011), no. 3, 309--339.

 
\bibitem{EFRV} E.~Edoli, S.~Fiorenzani, S.~Ravelli, T.~Vargiolu, \textit{Modeling and valuing make-up clauses in gas swing contracts}, Energy Economics 35 (2013), 58--73.
 
\bibitem{EV} E.~Edoli, T.~Vargiolu, \textit{Pricing of gas swing contracts: a viscosity solution approach}, preprint 2013.

\bibitem{BLN2} M.~Eriksson, J.~Lempa, T.~K.~Nilssen, \textit{Swing options in commodity markets: A multidimensional L\'evy diffusion model}, preprint  2013.

\bibitem{FS} W.~H.~Fleming, H.~M.~Soner, \textit{Controlled Markov Processes and Viscosity Solutions, Second Edition}, Springer, New York, 2006.

\bibitem{JRT} P.~Jaillet, E.~I.~Ronn, S.~ Tompaidis, \textit{Valuation of Commodity-Based Swing Options}, Management Science 50 (2004), no. 7, 909--921.

\bibitem{KR} N.~V.~Krylov, \textit{Controlled Diffusion Processes}, Springer, New York, 1980.

\bibitem{Loland} A.~L\o land, O.~Lindqvist, {\em Valuation of commodity-based swing options: a survey}, note SAMBA/38/80, Norwegian Computing Center (2008), {\tt http://publications.nr.no/swingoptionsurvey.pdf}. 

\bibitem{MS} M.~Motta, C.~Sartori, \textit{Finite fuel problem in nonlinear singular stochastic control}, SIAM J. Control Optim. 46 (2007), no. 4,  1180--1210.

\bibitem{MS2} M.~Motta, C.~Sartori, \textit{Minimum time with bounded energy, minimum energy with bounded time}, SIAM J. Control Optim. 42 (2003), no. 3, 789--809.

\bibitem{so} P.~Soravia, \textit{Viscosity solutions and optimal control problems with integral constraints}, Systems Control Lett. 40 (2000), no. 5, 325--335.

\end{thebibliography}
\end{document}